\documentclass[11pt, letterpaper, reqno]{amsart}
\usepackage[utf8]{inputenc} 	
\usepackage{microtype} 			
\usepackage{geometry} 
\usepackage{amsmath}			
\usepackage{amsthm}		 		
\usepackage{amssymb}	 
\usepackage{bm}					
\usepackage{mathrsfs}			
\usepackage{soul}
\usepackage{thmtools}
\usepackage{xcolor}				
\definecolor{darkred}{rgb}{0.5, 0.2, 0.2}
\definecolor{darkblue}{rgb}{0..2, 0.2, 0.7}
\definecolor{darkgreen}{rgb}{0..2, 0.45,.2}
\definecolor{mygreen}{rgb}{0.35, 0.66,.35}
\definecolor{purp}{rgb}{0.5, 0.1,.5}
\definecolor{bettercyan}{rgb}{0.1, 0.4, 0.7}
\usepackage{tikz}				
\usepackage{tikz-3dplot}
\usetikzlibrary{patterns}
\usetikzlibrary{arrows}
\usetikzlibrary{decorations.markings,decorations.pathmorphing,shapes}
\usepackage[all]{xy}			
\usepackage{graphicx}			
\usepackage{pgfplots}			
\usepackage{enumitem} 			
\usepackage{array}
\usepackage{lmodern}			
\usepackage[T1]{fontenc}		
\usepackage{cancel}
\usepackage{comment}
\usepackage[breaklinks=true, colorlinks=true, citecolor=bettercyan, linkcolor=darkred, urlcolor=magenta]{hyperref}
\usepackage{cleveref}
\usepackage[backend=bibtex, style=alphabetic, backref=true, firstinits=true, isbn=false, date=year]{biblatex}
\makeatletter
\renewbibmacro{in:}{}
\DeclareFieldFormat{pages}{#1}
\renewcommand*{\bibnamedash}{%
	\leavevmode\raise +0.6ex\hbox to 5.5ex{\hrulefill}.\space\space}

\InitializeBibliographyStyle{\global\undef\bbx@lasthash}

\newbibmacro*{bbx:savehash}{\savefield{fullhash}{\bbx@lasthash}}

\renewbibmacro*{author}{%
	\ifboolexpr{
		test \ifuseauthor
		and
		not test {\ifnameundef{author}}
	}
	{%
		\iffieldequals{fullhash}{\bbx@lasthash}
		{\bibnamedash\addcomma\space}
		{\printnames{author}}%
		\usebibmacro{bbx:savehash}%
		\iffieldundef{authortype}
		{}
		{%
			\setunit{\addcomma\space}%
			\usebibmacro{authorstrg}%
		}%
	}
	{\global\undef\bbx@lasthash}%
}
\makeatother

 \numberwithin{equation}{section}

\newtheorem{propositionx}{Proposition}[section]
\newenvironment{proposition}
{\pushQED{\qed}\propositionx}
{\popQED\endpropositionx}

\newtheorem*{theorem*}{Theorem}
\newtheorem{theorem}{Theorem}

\theoremstyle{remark}
\newtheorem{remarkx}[propositionx]{Remark}
\newenvironment{remark}
{\pushQED{\qed}\remarkx}
{\popQED\endremarkx}
\newtheorem{examplex}{Example}[section]
\newenvironment{example}
{\pushQED{\qed}\examplex}
{\popQED\endexamplex}
\newcommand{\desc}{\mathrm{de,sc}}

\newcommand{\WF}{\ensuremath{\mathrm{WF}}}

\newcommand{\bbB}{\mathbb{B}}
\newcommand{\bbC}{\mathbb{C}}

\newcommand{\bbM}{\mathbb{M}}
\newcommand{\bbN}{\mathbb{N}}
\newcommand{\bbO}{\mathbb{O}}
\newcommand{\bbP}{\mathbb{P}}

\newcommand{\bbR}{\mathbb{R}}
\newcommand{\bbS}{\mathbb{S}}

\newcommand{\bbZ}{\mathbb{Z}}

\newcommand{\calA}{\mathcal{A}}

\newcommand{\calC}{\mathcal{C}}
\newcommand{\calD}{\mathcal{D}}

\newcommand{\calF}{\mathcal{F}}

\newcommand{\calI}{\mathcal{I}}

\newcommand{\calK}{\mathcal{K}}

\newcommand{\calN}{\mathcal{N}}

\newcommand{\calR}{\mathcal{R}}
\newcommand{\calS}{\mathcal{S}}

\newcommand{\calV}{\mathcal{V}}

\newcommand{\frakm}{\mathfrak{m}}

\newcommand{\bfx}{\mathbf{x}}

\newcommand{\dd}{\,\mathrm{d}}
\makeatletter
\@namedef{subjclassname@2020}{\textup{2020} Mathematics Subject Classification}
\makeatother

\newlength\octheight
\newlength\octheightr
\newlength\arrowscale
\makeatletter
\@namedef{subjclassname@2020}{\textup{2020} Mathematics Subject Classification}
\makeatother

\title[The essential self-adjointness of the wave operator on radiative spacetimes]{The essential self-adjointness of the wave operator on radiative spacetimes}
\subjclass[2020]{Primary 35L05, 35P05. Secondary 58J47, 58J50.}

\date{October 8th, 2025 (Last update). December 6th, 2024 (First announced).}

\author{Qiuye Jia}
\email{Qiuye.Jia@anu.edu.au}
\address{Mathematical Sciences Institute, Australian National University, Canberra, ACT, Australia}

\author{Mikhail Molodyk}
\email{mam765@stanford.edu}
\address{Department of Physics and Stanford Institute for Theoretical Physics, Stanford University, California, USA}

\author{Ethan Sussman}
\email{ethan.sussman@northwestern.edu}
\address{Department of Mathematics, Northwestern University, Illinois, USA}
\begin{document}
	
\begin{abstract}
	We prove the essential self-adjointness of the d'Alembertian $\square_g$, allowing a larger class of spacetimes than previously considered, including those that arise from perturbing Minkowski spacetime by gravitational radiation. We emphasize the fact, proven by Taira in related settings, that all tempered distributions $u$ satisfying $\square_g u = \lambda u +f$ for $\lambda\in \bbC\backslash \bbR$ and $f$ Schwartz are Schwartz. The proof is fully microlocal and relatively quick given the ``de,sc-'' machinery recently developed by the third author.
\end{abstract}

\maketitle

\tableofcontents

\section{Introduction}

Consider an asymptotically flat metric $g$ on spacetime, which for simplicity we take to be diffeomorphic to $\bbR^{1,d}$, $d\geq 1$. (See below for the precise assumptions.) Then the d'Alembertian $\square_g$ is symmetric with respect to the $L^2=L^2(\bbR^{1,d},g)$-inner product, in the sense that 
\begin{equation}
\langle u, \square_g v \rangle_{L^2(\bbR^{1,d},g)}=\langle \square_gu,  v \rangle_{L^2(\bbR^{1,d},g)}
\end{equation}
for all $u,v\in \calS(\bbR^{1,d})$, where $\calS(\bbR^{1,d})$ denotes the set of Schwartz functions on spacetime.
It is of interest to show that $\square_g$, or more generally an $L^2(\bbR^{1,d},g)$-symmetric operator of the form 
\begin{equation} 
P=\square_g + \mathsf{m}^2 + L,\qquad \text{$\mathsf{m}\geq 0$ (or even $\mathsf{m}\in i\bbR$),} 
\end{equation} 
\begin{equation}
	\text{$L\in \operatorname{Diff}^1(\bbR^{1,d})$ with sufficiently decaying coefficients} 
\end{equation}  
is, when acting on the domain $C_{\mathrm{c}}^\infty(\bbR^{1,d})$ (or $\calS(\bbR^{1,d})$), essentially self-adjoint. For instance, if one wants to define the Feynman propagator as
\begin{equation}
(P-i0)^{-1}=\lim_{\varepsilon \to 0^+} (P - i \varepsilon )^{-1},
\label{eq:Feynman}
\end{equation}
then it is first necessary to make sense of $(P-i\varepsilon)^{-1}$, which, if $P$ is known to be essentially self-adjoint, is an immediate consequence of the functional calculus for self-adjoint operators. 
This line of inquiry was first followed by Derezi{\'n}ski and Siemssen in \cite{derezinski2018feynman, derezinski2019evolution, First}, which tackled e.g.\ the case when $g$ is stationary. Later works by Vasy \cite{vasy2020essential} and Nakamura and Taira \cite{nakamura2021selfadjoint-realprincipal, TairaSolo, nakamura2023remark, nakamura2023self-adjoint-asy-static} have tackled some cases of asymptotically Minkowski and asymptotically static $g$. Relevant to this paper is the fact that Vasy's arguments are microlocal and based on the Parenti--Shubin--Melrose ``scattering (sc)-calculus'' $\Psi_{\mathrm{sc}}(\bbR^{1,d})$, which consists of certain (but not all) polynomially weighted pseudodifferential operators in H\"ormander's standard calculus $\Psi_\infty(\bbR^{1,d})$. 
This choice of calculus limits the spacetimes that Vasy can consider.

The proof of essential self-adjointness begins with the easy consequence of the theory of deficiency indices (see \S\ref{sec:deficiency}) that all one needs to do is prove that 
\begin{equation} 
u\in L^2, Pu=\pm i u \Rightarrow u=0.
\end{equation} 
In this way, essential self-adjointness is reduced to a statement regarding the solvability theory of the operator $P\pm i$.
The symmetry of $P$ acting on $\calS=\calS(\bbR^{1,d})$ yields, via the usual argument that symmetric operators have only real eigenvalues on their domains, that $u\in \mathcal{S}, Pu=\pm i u \Rightarrow u=0$.
Indeed, if $u\in \calS$ and $Pu = \pm i u$, then (with $\lVert \cdot \rVert$ denoting the $L^2$ norm)
\begin{equation} 
\pm i \lVert u \rVert^2=\langle Pu, u \rangle_{L^2} = \langle u,Pu \rangle_{L^2} = \mp i \lVert u \rVert^2,
\label{eq:symmetry_calc}
\end{equation} 
which is only possible if $u=0$. So, in order to prove the essential self-adjointness of $P$, it suffices to prove that 
\begin{equation}  
u\in L^2, Pu=\pm i u \Rightarrow u \in \calS.
\end{equation}
Taira \cite[Thm.\ 1.6]{TairaSolo} has proven the stronger result (in some level of generality) that $u\in L^2, f\in \calS, Pu=\pm i u+f \Rightarrow u \in \calS$. 
This feature of the solvability theory of $P$, which we will call the Schwartz-to-Schwartz mapping property of $(P\mp i)^{-1}$, is surprisingly delicate in the variable-coefficient case. (It is rather trivial in the constant-coefficient case --- see \S\ref{sec:constant_coeff}.) 

Microlocally, it is straightforward to prove that the Schwartz-to-Schwartz property holds away from the intersection of null infinity and fiber infinity (i.e.\ infinite frequency) in the appropriate phase space. In other words, if $Pu\mp iu \in \calS$, and if $A\in \Psi_{\mathrm{sc}}^{0,0}(\bbR^{1,d})$ is an operator in the Parenti--Shubin calculus whose essential support is disjoint from the portion of fiber infinity over null infinity, then $Au\in \calS$. However, it is difficult to remove the ``microlocalizer'' $A$.  What Vasy does in  \cite{vasy2020essential} is prove a weaker and delicate borderline radial point estimate, the conclusion of which is that $u$ has some finite amount of Sobolev regularity. Vasy then proves that this exact amount of Sobolev regularity is sufficient to make the computation \cref{eq:symmetry_calc} go through. In this way, Vasy proves essential self-adjointness without proving the Schwartz-to-Schwartz property of $(P\mp i)^{-1}$.

One goal of this note is to show that, if,  instead of the Parenti--Shubin calculus, one uses the double edge-scattering (de,sc-) pseudodifferential calculus $\Psi_{\mathrm{de,sc}}$ recently introduced in \cite{sussman2023massive} by one of the authors in order to analyze the asymptotics of solutions of the Klein--Gordon equation near null infinity, then standard microlocal methods suffice to prove the Schwartz-to-Schwartz mapping property, obviating the need for Vasy's borderline estimate. In fact, we will prove that $u\in \calS', f\in \calS, Pu=\pm i u+f \Rightarrow u \in \calS$.
Moreover, we are able to prove a natural generalization of the mapping property that applies to $f$ with only a finite amount of Sobolev regularity or decay --- see \Cref{thm:variable-order}. This hierarchy of mapping properties is one respect in which our result (and not just the method) is not already contained in \cite{nakamura2021selfadjoint-realprincipal,  TairaSolo, nakamura2023remark}.

The de,sc-calculus is well-suited to this problem for the same reason it is well-suited to analyze massive waves near null infinity: the oscillations present in solutions of the Klein--Gordon equation lie at finite de,sc-frequency. We explain this perspective in connection with the asymptotics problem in \cite[\S1]{sussman2023massive}. This finitude of frequency is what allows us to replace Vasy's borderline estimate with \emph{an elliptic} estimate in which the imaginary term in $P \pm i$ provides ellipticity. The Schwartz-to-Schwartz mapping property is then proven by propagating the control provided by ellipticity throughout the characteristic set of $P$. Heuristically, $P\pm i$ is a Klein--Gordon operator with ``complex mass;'' it should therefore not be surprising that tools developed to understand the Klein--Gordon equation are useful, even when the original operator $P$ is the massless wave operator $\square_g$.

A second goal is to extend the essential self-adjointness result to a class of metrics with more singular behavior at null infinity than those considered by Nakamura--Taira and Vasy. The simplest sort of asymptotically Minkowski spacetimes have metrics whose coefficients are well-behaved on the radial compactification\footnote{This means using the embedding $\bbR^{1+d}\backslash \text{origin} \ni (t,\bfx)\mapsto ( (t^2+\lVert \bfx \rVert^2)^{-1/2}, (t,\bfx)/(t^2+\lVert \bfx \rVert^2)^{1/2} )  \in [0,\infty)\times \bbS^d$. Identifying $\bbR^{1+d}\backslash \text{origin}$ with the image of this embedding, $\{0\}\times \bbS^d$ can be thought of as the ``sphere at infinity,'' called $\infty\bbS^d$ above. } 
\begin{equation} 
	\bbM=\overline{\bbR^{1,d}} = \bbR^{1,d}\sqcup \infty \bbS^d
	\label{eq:bbM-def}
\end{equation} 
of Minkowski spacetime (at least if we ignore the north/south pole of the compactification). It is to such metrics that Vasy's and Nakamura--Taira's analysis applies. However, on solving the Einstein field equations or other quasilinear wave equations, a typical phenomenon is radiation. Radiation, which is coupled to and therefore has an effect on the metric, tends to null infinity. 
Consequently, the metric is not expected to be well-behaved on the radial compactification $\bbM$ but rather on another compactification, the ``octagon''
\begin{equation}
	\bbO = [\overline{\bbR^{1,d}} ; \text{null infinity};1/2] \hookleftarrow \bbR^{1,d}
	\label{eq:bbO-def}
\end{equation}
that results from blowing up null infinity.\footnote{The `1/2' refers to a change of smooth structure at the front face of the blowup, null infinity. This was done in \cite{sussman2023massive} to simplify some aspects of the microlocal analysis. Specifically, the oscillations present in solutions of the PDE oscillate like $\exp(i (t^2-r^2)^{1/2})$; if we fix $v=t-r$ and send $r\to\infty$, this only oscillates like $\exp(i (vr)^{1/2})$, which is oscillating more slowly than the $\exp(i \sigma r)$ waves the sc-calculus is designed to detect. 
For this reason, it is convenient to choose the smooth structure such that $1/\sqrt{r}$ becomes a boundary-defining-function in the interior of null infinity, instead of $1/r$. 
Besides this effect, the choice smooth structure can be ignored by the reader.} The resultant manifold-with-corners $\bbO$ has five boundary hypersurfaces (unless $d=1$, in which case $\bbO$ is literally an octagon), two of which are the timelike caps, one of which is spacelike infinity, and the other two of which, the front faces of the blowup, correspond to the boundaries of the Penrose diagram of Minkowski spacetime. 
Indeed, Hintz and Vasy have proven rigorously in \cite{H-V-stability} that solving the Einstein vacuum equations on $\bbR^{1,d}$ for initial data differing only slightly from the flat case
results in a metric polyhomogeneous -- i.e.\ well-behaved, in a precise sense -- on $\bbO$.\footnote{If the perturbation does not impart a mass to the spacetime, then the compactification $\bbR^{1,d}\hookrightarrow \bbO$ in the Hintz--Vasy theorem is just \cref{eq:bbO-def}. However, if the perturbation \emph{does} impart a mass to the spacetime, then the compactification must be altered slightly to one involving the mass (cf.\ \Cref{ex:astro}). The target is still $\bbO$, but the inclusion $\bbR^{1,d}\hookrightarrow \bbO$ must be adjusted. The physical reason for this is that null geodesics are not, near null infinity, approximately lines of constant $t-r$, but rather constant $t-r_*$, where $r_*$ is the Eddington--Finkelstein tortoise coordinate. Changing perspective, \cite{H-V-stability} tells us that, regardless of whether or not the perturbed spacetime is massive, the perturbed spacetime is \emph{isometric}, via an isometry close to the identity, to one on $\bbR^{1,d}$ which is polyhomogeneous with respect to the compactification \cref{eq:bbO-def} and close to Minkowski.
This is what matters for the analysis here to be applicable.}
The manifold-with-corners $\bbO$ is therefore the natural one vis-a-vis the analysis of relativistic wave equations.

One attractive feature of $\Psi_{\mathrm{de,sc}}$ is that it permits microlocal analysis directly on $\bbO$, including PDEs whose coefficients are singular on $\bbM$. The sc-calculus is not well-suited for this purpose. Another calculus similar to the de,sc-calculus is the edge-b (e,b-) calculus $\Psi_{\mathrm{e,b}}$, developed by Melrose--Vasy--Wunsch \cite{MVW} and deployed recently by Hintz--Vasy in \cite{HintzVasyScriEB} to study the massless wave equation on $\bbO$. The massless wave equation was previously analyzed on $\bbO$ by Baskin--Vasy--Wunsch in \cite{BaskinVasyWunsch, BaskinVasyWunscher}; however, in these papers, the metric is required to be well-behaved \emph{on $\bbM$}. It is the more sophisticated microlocal tools developed in the later work that allowed the handling of radiative spacetimes, which are only well-behaved on $\bbO$.

Thus, we prove: 
\begin{theorem}
	Suppose that 
	\begin{itemize}
		\item as discussed in \S\ref{sec:statement}, $g$ is non-trapping and well-behaved on $\bbO$, in the sense that 
		\begin{equation} 
		g - \eta \in \rho^{\epsilon}S^{\mathsf{0}}(\bbO ; {}^{\mathrm{de,sc}} \operatorname{Sym}^2\overline{T}^* \bbO)
		\label{eq:metric_ass_gen}
		\end{equation} 
		 holds 
		 for some $\epsilon>0$, 
		 where $\eta$ is the exact Minkowski metric and $\rho$ is a total boundary defining function of $\bbO$ (see \cref{eq:tdf})\footnote{The exact smooth structure on $\bbO$ does not matter here, so the reader may pretend that $\rho = 1/(1+t^2+r^2)^{1/2}$ without changing the class of metrics that \cref{eq:metric_ass_gen} allows.};
		\item $A$ is an $L^2(\bbR^{1,d},g)$-symmetric first-order operator\footnote{Meaning that $\langle Au,v \rangle_{L^2(\mathbb{R}^{1,d},g)}=\langle u,Av \rangle_{L^2(\mathbb{R}^{1,d},g)}$ for all Schwartz $u,v$. } sufficiently decaying in the sense that 
		\begin{equation} 
		A \in S\operatorname{Diff}_{\mathrm{de,sc}}^{1,-\mathsf{1}},
		\label{eq:A_ass}
		\end{equation} 
	\end{itemize}
	and $\mathsf{m}^2 \geq 0$. 
	Then, the differential operator $P=\square_g + \mathsf{m}^2+A$ is essentially self-adjoint on $C_{\mathrm{c}}^\infty(\bbR^{1,d})$ with respect to the $L^2(\bbR^{1,d},g)$-inner product.
	\label{thm:main_0}
\end{theorem}
\begin{remark}
	If $\mathsf{m}\in i \mathbb{R}$ (the ``tachyonic case''), then the topological structure of the characteristic set in the interior of the fibers changes. This actually does not affect our arguments, because the relevant propagation results are all at fiber infinity. Nevertheless, we have chosen to restrict attention to the physically relevant $\mathsf{m}\in \mathbb{R}$ case, for simplicity.
\end{remark}
\begin{remark}
	Because the essential self-adjointness of $P$ is equivalent to that of $\square_g+ A$, there is no loss of generality in taking $\mathsf{m}=0$. However, since we make thorough use of the results in \cite{sussman2023massive}, in which $\mathsf{m}$ must be $>0$, we will leave this parameter be, for clarity's sake. (This is not mathematically necessary.)
\end{remark}

\begin{remark}
	Here, the `$S$' in ``$S\operatorname{Diff}_{\mathrm{de,sc}}^{1,-\mathsf{1}}$'' refers to the fact that the coefficients of $A$ are only required to be conormal, not necessarily smooth, on $\bbO$; see \S\ref{subsec:calc} for details. 
	
	The above also holds if 
	\begin{equation} 
		A \in \rho^\epsilon S\operatorname{Diff}_{\mathrm{de,sc}}^{1,\mathsf{0}}
	\end{equation} 
	for some $\epsilon>0$. This complicates some principal symbols below, but the added terms do not affect the argument.
\end{remark}
\begin{remark}
	Our assumptions on the metric are coming from \cite[Thm. 1.8]{H-V-stability}, where the symbolic character of the metric is proven assuming symbolic regularity of initial data.  
	The major difference between this and spacetimes arising in the work of Christodoulou--Klainerman \cite{CK} is that the latter only require control on finitely many derivatives. 
	However, every estimate in this paper makes use of only \emph{finitely} many symbolic seminorms.\footnote{For example, in order to prove the $L^2$-boundedness of de,sc-$\Psi$DOs, one only needs a finite-order parametrix construction to apply the square root trick.}
	So, the conclusion should only require control on finitely many derivatives of the metric. 
\end{remark}

It is straightforward to construct metrics satisfying the hypotheses above that are short-range perturbations of Minkowski. Realistic astrophysical examples typically involve long-range perturbations, and these require a bit of care to fit into our framework. But, they do fit --- see \S\ref{sec:examples}. 
See \cite{Bicak0, BicakI, BicakII} for many explicit examples of radiating spacetimes. Some, e.g.\ the pp-wave spacetimes, cannot be fit into the framework here, since a plane wave is a very large perturbation of the metric. Other models, such as the Vaidya example discussed below (in \S\ref{sec:examples}), do in their domain of physical relevance, under appropriate assumptions on the global structure of the metric. 
While explicit examples may be few in number, there exists a large literature regarding the asymptotic structure of generic radiating spacetimes, originating in the works of Bondi et al.\ \cite{Bondi}, Penrose et.\ al \cite{NewmanPenrose, Penrose}, and Sachs \cite{Sachs}. See \cite{CruscielReview}\cite{Puzzle} for citations of some of the work that followed. For rigorous mathematical work on small perturbations of Minkowski spacetime, see \cite{CK}\cite{LindbladRodnianski} and the already cited \cite{H-V-stability}.

However, $\bbO$, or some closely related construction, seems to have appeared in the physics literature only recently; indeed, Comp{\`e}re, Gralla, and Wei \cite{Puzzle} have introduced what they call the ``puzzle piece diagram'' of spacetime (essentially compactification by $\bbO$), about which they claim:
\begin{quote}
	\textit{It is perhaps surprising that a fully general framework for macroscopic gravitational scattering does not already exist, given the tremendous effort devoted to the study of compact object interactions. Indeed, a framework almost exists, in that various asymptotic formalisms need only be tweaked and stitched together.} \cite[p.\ 3]{Puzzle}
\end{quote} 
``Stitched together'' means combining the various asymptotic regimes considered previously into a single compactification. Ashtekar et al.\ \cite{AshtekarI, AshtekarIII, AshtekarIV} had previously considered the corner between spacelike and null infinity. 
Instead of working on $\bbO$, physicists often work directly on the Penrose diagram $\bbP$ (see \Cref{fig:o}), in which timelike and spacelike infinity have been collapsed to various points. A metric is well-behaved (more precisely, symbolic) near points in the interior of null infinity if and only if it is well-behaved near the corresponding points at null infinity in the Penrose diagram. Well-behavedness on $\bbO$ is a reasonable criterion for what it means for a metric to be well-behaved near the \emph{corners} of the Penrose diagram. Besides this, there are several reasons for working on $\bbO$ rather than $\bbP$: 
\begin{itemize}
	\item unless $d=1$ (in which case we are working in 1+1D), the Penrose diagram is not a manifold-with-corners locally modeled smoothly on $2^k$th-ants of $\bbR^{1+d}$, which complicates analysis on it. For example, if $d=2$, the Penrose diagram is, near its timelike corners, a solid cone, not an octant of $\bbR^3$.
	\item  Even if the metric is well-behaved already on the Penrose diagram, solutions of massive wave equations will not be. For instance,  it is necessary to blow up timelike infinity in the Penrose diagram in order to resolve the asymptotics of solutions of the Klein--Gordon equation, as can be seen already in the free case \cite{HormanderNL}. This is the perspective of Comp\`ere et al. \cite{Puzzle}.
\end{itemize}
Related to the first point, microlocal tools do not directly apply to the Penrose diagram. Consequently it can be worthwhile to pass to $\bbO$ even when studying the massless wave equation \cite{HintzVasyScriEB}, despite the fact that the second point above does not apply to massless waves. However, if the reader so desires, they may assume that the metric is well-behaved on $\bbP$ (whatever precise definition one gives for ``well-behavedness'' near the corners of $\bbP$, it should imply being symbolic on $\bbO$) and view the passage to $\bbO$ as a technical device.

\begin{figure}[t!]
	\begin{tikzpicture}[scale=.75]
	\begin{scope}[shift={(-6.5,0)}]
	\draw[dashed] (-3.4,-3) -- (-3.4,3) -- (2.5,3) -- (2.5,-3) -- cycle;
	\node (0) at (-2.9,2.6) {$\bbM$}; 
	\filldraw[fill=gray!10] (0,0) circle (65pt);
	\draw (0,1) ellipse (58pt and 4pt);
	\draw[dashed] (0,0) ellipse (65pt and 4pt);
	\draw[dashed] (0,1.75) ellipse (40pt and 4pt);
	\draw[dashed] (0,-1.75) ellipse (40pt and 4pt);
	\draw (0,-1) ellipse (58pt and 4pt);
	\fill (2.05,1) circle (2pt);
	\fill (2.05,-1) circle (2pt);
	\fill (-2.05,1) circle (2pt);
	\fill (-2.05,-1) circle (2pt);
	\node (1) at (0,2.6) {Ff};
	\node (1) at (-2.7,1) {nFf};
	\node (1) at (-2.7,0) {Sf};
	\node (1) at (-2.7,-1) {nPf};
	\node (1) at (0,-2.6) {Pf};
	\end{scope}
	\draw[dashed] (-3.8,-3) -- (-3.8,3) -- (3,3) -- (3,-3) -- cycle;
	\node (0) at (-3.35,2.6) {$\bbO$}; 
	\begin{scope}[shift={(.2,0)}]  
	\coordinate (one) at (.413\octheight,\octheight) {};
	\coordinate (two) at (\octheight,.413\octheight) {}; 
	\coordinate (three) at (\octheight,-.413\octheight) {};
	\coordinate (four) at (.413\octheight,-\octheight) {};
	\coordinate (five) at (-.413\octheight,-\octheight) {};
	\coordinate (six) at (-\octheight,-.413\octheight) {};
	\coordinate (seven) at (-\octheight,.413\octheight) {}; 
	\coordinate (eight) at (-.413\octheight,\octheight) {};
	\draw[fill=gray!10] (one) -- (two) -- (three) -- (four) -- (five) -- (six) -- (seven) -- (eight) -- cycle;
	\node (Ff) at (.1,1.2\octheight) {$\mathrm{Ff}=i^+$};
	\node (nFf) at (-1.2\octheight,.81\octheight) {$\mathrm{nFf}=\calI^+$};
	\node (Sf) at (-1.42\octheight,0) {$\mathrm{Sf}=i^0$};
	\node (Pf) at (.1,-1.2\octheight) {$\mathrm{Pf}=i^-$};
	\node (nPf) at (-1.2\octheight,-.81\octheight) {$\mathrm{nPf}=\calI^-$};
	\draw[dashed] (0,-.75\octheight) ellipse (39pt and 4pt);
	\begin{scope}
		\clip (-\octheight,-.5\octheight) rectangle (+\octheight,-.413\octheight);
		\draw (0,-.413\octheight) ellipse (59.5pt and 4pt);
	\end{scope}
	\draw[dashed] (0,-.413\octheight) ellipse (59pt and 4pt);
	\draw[dashed] (0,0) ellipse (60pt and 4pt);
	\begin{scope}
		\clip (-\octheight,+.3\octheight) rectangle (+\octheight,+.413\octheight);
		\draw (0,+.413\octheight) ellipse (59.5pt and 4pt);
	\end{scope}
	\draw[dashed] (0,+.413\octheight) ellipse (59pt and 4pt);
	\draw[dashed] (0,+.75\octheight) ellipse (39pt and 4pt);
	\draw[fill=gray!10] (0,\octheight) ellipse (24.5pt and 4pt);
	\draw[fill=gray!20] (0,-\octheight) ellipse (24.5pt and 4pt);
	\end{scope} 
	\begin{scope}[shift={(6,0)}]
	\draw[dashed] (-2.9,-3) -- (-2.9,3) -- (3,3) -- (3,-3) -- cycle;
	\node (0) at (-2.5,2.6) {$\bbP$}; 
	\filldraw[fill=gray!10] (0,2.25) -- (2.25,0) -- (0,-2.25) -- (-2.25,0) -- cycle;
	\draw[dashed] (0,1) ellipse (35pt and 4pt);
	\draw[dashed] (0,0) ellipse (62pt and 4pt);
	\draw[dashed] (0,-1) ellipse (35pt and 4pt);
	\node (1) at (0,2.5) {Ff};
	\node (1) at (1.5,1.4) {nFf};
	\node (1) at (2.6,0) {Sf};
	\node (1) at (1.5,-1.4) {nPf};
	\node (1) at (0,-2.5) {Pf};
	\end{scope}
	\end{tikzpicture}
	\caption{
		(\textit{Left}) The radial compactification $\bbM$ of Minkowski spacetime, (\textit{middle}) the manifold-with-corners
		$\bbO$, and (\textit{right}) the Penrose diagram $\bbP$ of Minkowski spacetime, all with labeled faces: $\mathrm{Ff}$ is future timelike infinity, $\mathrm{nFf}$ is future null infinity, $\mathrm{Sf}$ is spacelike infinity, $\mathrm{nPf}$ is past null infinity, and $\mathrm{Pf}$ is past timelike infinity. The compactification $\bbO$ can be constructed by blowing up the corners of the Penrose diagram or blowing up null infinity in the radial compactification. In this sense, $\bbO$ is the simplest compactification of Minkowski spacetime refining both the radial and Penrose compactifications. Physicists use the notation $i^\pm,i^0,\calI^\pm$ to label the various asymptotic regimes. We have provided the translation in the middle panel. In all three figures, time is increasing vertically and $r$ is increasing moving away from the vertical midline.}
	\label{fig:o}
\end{figure}

It is worth pointing out that we do not allow perturbations of the Minkowski metric of the form $\eta+\delta \eta$ for $\delta \eta$ static.\footnote{In \cite{nakamura2023self-adjoint-asy-static}, Nakamura and Taira allow time-dependent perturbations of static spacetimes, but only with compact Cauchy hypersurfaces.} The physical significance of this limitation will be seen in \S\ref{sec:examples}.
From our perspective, asymptotically flat metrics with static perturbations are badly behaved near the north/south pole of the radial compactification of Minkowski spacetime, and neither the Parenti--Shubin calculus $\Psi_{\mathrm{sc}}$ nor $\Psi_{\mathrm{de,sc}}$ can tolerate this. Fortunately, there exist extensions of $\Psi_{\mathrm{sc}}$ designed to handle such perturbations. For example, Baskin, Doll, and Gell-Redman have used the ``3-body-sc''-calculus to analyze the Klein--Gordon equation with an asymptotically static potential \cite{3bodysc,3bodysc-Feynman}. We expect that it is possible to use their tools to extend the results of this paper to allow static perturbations of the metric. Because the problems we are concerned with here are delicate only near null infinity, and because the north/south pole of the radial compactification of spacetime is disjoint from null infinity, we do not believe that any major obstacles exist to combining the de,sc- and 3-body-sc- analyses. Thus, a ``de,sc,3-body-'' analysis should be possible, carried out with a corresponding pseudodifferential calculus. The result would be an extension of our main theorems.

Let us close this introduction by noting that, in recent work, Dang, Vasy, and Wrochna \cite{DangVasyWrochna2024dirac} apply a second microlocal approach to develop a functional calculus for squares of Dirac operators. Their tools, applied to the problem at hand, yield an independent and arguably simpler proof of the Schwartz-to-Schwartz mapping property of $(P\pm i)^{-1}$ but do not allow more general metrics than those considered already by Vasy. This follows previous work by Dang and Wrochna \cite{DW-residues, DW-complex} developing a Lorentzian spectral action principle for gravity, for which one needs complex powers of $P\pm i \varepsilon$ to be well-defined, as  they are if $P$ is essentially self-adjoint. In the context of variational principles for gravity, it is natural to work with classes of metrics which include Einstein perturbations of Minkowski space. Our methods are able to handle these (typically radiative) spacetimes.\footnote{We thank Micha{\l} Wrochna for bringing this perspective to our attention.}

\begin{remark}[Alternative definition of Feynman propagators]
	Below, the Schwartz-to-Schwartz mapping property of $(P\pm i)^{-1}$ is proven using propagation estimates for $P\pm i$, where the sign of the imaginary part determines direction of propagation. Concatenating propagation estimates for $P$ itself as in \cite{sussman2023massive} (where now one can choose the direction of propagation independently in the two sheets of the characteristic set, resulting in four possible choices) yields regularity results similar to \Cref{thm:variable-order}. This leads to four Fredholm realizations of $P$, following the general approach to Fredholm theory for non-elliptic equations originating in \cite{Vasy-AH-KdS}. If we propagate in the same direction as required for $P-i$, the resulting inverse of $P$ is the Feynman propagator --- cf. \cref{eq:Feynman}. This approach to defining Feynman propagators is taken by Gell-Redman, Haber, and Vasy \cite{GR-H-V-Feynman}, working on asymptotically Minkowski spacetimes which are well-behaved on $\bbM$. For massive waves, the analogous construction on the more general spacetimes considered here is the subject of the work \cite{M-V-Feynman}. The paper \cite{3bodysc-Feynman} already mentioned carries out this construction for massive waves on non-radiative spacetimes but allowing for an asymptotically static potential. Cf.\ \cite{GerardWrochna1, GerardWrochna2} for other results on Fredholm theory in the massive case and \cite{DG-propagators} for a discussion of the relationship between different notions of Feynman propagators.
\end{remark}

\begin{remark}[Freedom in choice of compactification]
As discussed in \cite[Section 6]{M-V-Feynman}, the Schwartz-to-Schwartz mapping property of $(P+\lambda)^{-1}$ for $\lambda\notin\bbR$ implies a sort of rigidity property for the resolvent itself and the Feynman/anti-Feynman propagator as its limit as $\Im \lambda \to  0^\pm $. While it is convenient to state assumptions on the asymptotic structure of spacetime by starting with a compactification from the outset, there may be many ways to attach the boundary to the spacetime so that the metric satisfies the required assumptions. Therefore, one may wonder to what extent any constructions are determined by the metric or depend on the choice of compactification. The Schwartz-to-Schwartz mapping property of the resolvent shows that, as long as the Schwartz spaces of two different compactifications match, the resolvents and Feynman propagators obtained from the de,sc-analysis on these compactifications match as well.
\end{remark}

\begin{remark}[e,b vs.\ de,sc]
	The operator $\square+\lambda$ can be analyzed as either as a e,b- operator or a de,sc- operator. However, only one is useful. If $\lambda\neq 0$, then $\square+\lambda$ is degenerate as an e,b-operator (because $\square$ has decay order $-2$ at null infinity as an edge operator, whereas $\lambda$ is order $0$, so there is a mismatch at the corner of the e,b-cotangent bundle), however the de,sc-analysis goes through. On the other hand, if $\lambda=0$, then this issue is not present, and moreover the de,sc-analysis degenerates for other reasons (the same reason why the sc-analysis of the Euclidean Laplacian $\triangle$ degenerates). This is why Hintz--Vasy used in \cite{HintzVasyScriEB} the e,b-calculus to analyze the massless wave equation, whereas Sussman used in \cite{sussman2023massive} the de,sc-calculus to analyze the Klein--Gordon equation.  
	
	In the present paper, the key value of $\lambda$ is $\lambda=\mathsf{m}^2\pm i$. This is nonzero even when $\mathsf{m}=0$, so the de,sc-calculus is used.
\end{remark}

\subsection*{Acknowledgements}

We are grateful to Dean Baskin, Jan Derezi\'nski, Kouichi Taira, Micha{\l} Wrochna, Andr\'as Vasy, and the anonymous referees for comments. We also thank the organizers of MAQD 2024 at Northwestern University. It was at this conference that this work began, inspired by a talk given by Shu Nakamura on his work on this problem. Q.J. is supported by the Australian Research Council through grant \texttt{FL220100072}. M.M. is supported by the National Science Foundation under grant number \texttt{PHY-2310429}. E.S.\ is supported by the National Science Foundation under grant number \texttt{DMS-2401636}.

\section{Some setup}
\label{sec:statement}

\subsection{Statement of main theorem}

By asymptotically Minkowski metric, we mean a Lorentzian metric $g$ on $\bbR^{1,d}$ such that \cref{eq:metric_ass_gen} holds. Here, 
\begin{equation}
\rho=\prod_{\mathrm{f}} \rho_{\mathrm{f}}, \label{eq:tdf}
\end{equation}
where the product is over the boundary hypersurfaces. 
In \cref{eq:metric_ass_gen}, $S^{\mathsf{0}}$ is the space of zeroth-order conormal symbols on $\bbO$, i.e.\ functions lying in $L^\infty$ after the application of arbitrarily many vector fields tangent to the boundary of $\bbO$.\footnote{See \cite{CruscielReview} for forceful justification for considering symbolic, or at least fairly general polyhomogeneous, metrics, rather than just those satisfying \cref{eq:metric_ass}.}
\begin{remark}
	The assumption \cref{eq:metric_ass_gen} is weaker than the corresponding assumption in \cite{sussman2023massive},
	\begin{equation}
		g - \eta \in \rho^2 C^\infty(\bbO ; {}^{\mathrm{de,sc}} \operatorname{Sym}^2\overline{T}^* \bbO).
		\label{eq:metric_ass}
	\end{equation}
	In this previous work, the goal was the production of asymptotic expansions, so stronger assumptions were convenient, but only used in \cite[\S7]{sussman2023massive}.\footnote{The precise statements in \cite[\S4]{sussman2023massive} also rely on the stronger assumption \cref{eq:metric_ass}, but only the weak dynamical structure they encode is required for \cite[\S5, \S6]{sussman2023massive}. This is one thing that \cite{M-V-Feynman} shows.} In work by the second author and Vasy \cite{M-V-Feynman}, the germane portion of the analysis in \cite[\S5, \S6]{sussman2023massive} is extended to metrics satisfying only \cref{eq:metric_ass_gen}. The e,b-analysis by Hintz and Vasy of massless waves near null infinity \cite{HintzVasyScriEB} was also carried out in this (and even greater) generality. Solutions of the Einstein vacuum equations with initial data close to those of Minkowski space generally satisfy \cref{eq:metric_ass_gen} but not the stronger \cref{eq:metric_ass}; see \cite[Theorem 7.1]{H-V-stability}. 
	
	In this paper, we use only results in \cite[\S5]{sussman2023massive}, not \cite[\S7]{sussman2023massive}, so we are permitted to assume only \cref{eq:metric_ass_gen}.
\end{remark}

In addition, we assume that $g$ is non-trapping, in the sense that every null geodesic asymptotes to the boundary of $\bbO$. Since $g$ is asymptotically flat, this just amounts to a restriction on the geodesic flow in the interior.

Note that $L^2(\bbR^{1,d},g) = L^2(\bbR^{1,d})$ at the level of topological vector spaces. So, unless the specific inner product is relevant, we will just write $L^2$.

We now consider differential operators $P$ of the form specified in \Cref{thm:main_0}: $P=\square_g + \mathsf{m}^2 + A$, where $\mathsf{m}\geq 0$, $\square_g$ is the d'Alembertian, and $A$ is as in the theorem. 
(The theorem is not less interesting if $A=0$.) Our main result is:

\begin{theorem}[Taira's mapping property, for radiating metrics] 
	Fix \emph{nonreal} $\lambda\in \bbC$. Then, if $u\in \calS'(\bbR^{1,d})$ solves $Pu +\lambda u = f$ for 
	$f\in \mathcal{S}$, we have $u\in \mathcal{S}$.
	\label{thm:main}
\end{theorem}

The proof is in the next section, \S\ref{sec:proof}. Note that \Cref{thm:main_0} follows immediately as a corollary, as discussed in the introduction.

\subsection{The de,sc-calculus and related objects}
\label{subsec:calc}
We refer to \cite{sussman2023massive} for a full description of $\Psi_{\mathrm{de,sc}}$. We refer especially to the introduction of that paper, which was intended to be readable. 
For the typical reader, whom we expect to be acquainted with microlocal analysis but not this particular calculus, we provide a brief summary here. In a few sentences: $\Psi_{\mathrm{de,sc}}$ is a calculus of pseudodifferential operators that arise via quantizing symbols on a particular compactification ${}^{\mathrm{de,sc}}\overline{T}^* \bbO$, 
\begin{equation}
	\underbrace{\bbB^{1+d}}_{\text{fiber}}\times  \underbrace{\bbO}_{\text{base}} \cong {}^{\mathrm{de,sc}}\overline{T}^* \bbO \hookleftarrow T^* \bbR^{1,d}
	\label{eq:desc_phase}
\end{equation}
of the cotangent bundle $T^* \bbR^{1,d}$ of spacetime. This compactification 
is a ball bundle over the octagonal compactification $\bbO\hookleftarrow \bbR^{1,d}$ of spacetime; distinguishing different base points at null infinity allows a more refined analysis than would otherwise be possible. The de,sc-calculus is fully symbolic, which means that, like the sc-calculus,\footnote{But unlike the b- or edge- calculi.} operators therein are captured by principal symbols modulo compact errors (operators which are \emph{both} smoothing \emph{and} decay inducing). Consequently, the de,sc-calculus works very similarly to the sc-calculus. In fact, the two are canonically identifiable away from null infinity --- this is what the ``sc'' in ``de,sc'' refers to. The ``de'' refers to the particular boundary fibration structure at null infinity, the \emph{double edge} structure, which is reverse-engineered from the formula for the d'Alembertian in lightcone coordinates.\footnote{The double edge structure and its associated pseudodifferential calculus was introduced by Lauter--Moroianu \cite{LauterMoroianu}. Their motivation seems to be unrelated to that here.} The double edge structure is related to the edge structure in the same way that the sc- structure is related to the b- structure. This is why the de- structure is fully symbolic.

We will use liberally notation introduced in \cite[\S1, \S2]{sussman2023massive}. 
For example, for $m\in \bbR, \mathsf{s} = (s_{\mathrm{Pf}},s_{\mathrm{nPf}},s_{\mathrm{Sf}},s_{\mathrm{nFf}},s_{\mathrm{Ff}} ) \in \bbR^5$, 
\begin{equation}
H_{\mathrm{de,sc}}^{m,\mathsf{s}} = \{u\in \calS'(\bbR^{1,d}) : Lu\in L^2(\bbR^{1,d})\text{ for all }L\in \Psi^{m,\mathsf{s}}_{\mathrm{de,sc}}  \}
\label{eq:Sobolev_def}
\end{equation}
is the scale of Sobolev spaces associated to the calculus $\Psi_{\mathrm{de,sc}}$ of de,sc-pseudodifferential operators.
Denoting\footnote{When the $s_{\mathrm{f}}$ are made \emph{variable}, the right-hand side of \cref{eq:tbdf} becomes a function on phase space, and $\rho^{-\mathsf{s}}$ is interpreted as its quantization.} 
\begin{equation}
	\rho^{-\mathsf{s}} = \rho_{\mathrm{Pf}}^{-s_{\mathrm{Pf}} } \rho_{\mathrm{nPf}}^{-s_{\mathrm{nPf}} }\rho_{\mathrm{Sf}}^{-s_{\mathrm{Sf}} }\rho_{\mathrm{nFf}}^{-s_{\mathrm{nFf}} }\rho_{\mathrm{Ff}}^{-s_{\mathrm{Ff}} }
	,
	\label{eq:tbdf}
\end{equation}
where 
$\rho_{\mathrm{f}} \in C^\infty(\bbO)$ is a boundary-defining-function for the boundary hypersurface $\mathrm{f}$, the space
$\smash{H_{\mathrm{de,sc}}^{m,\mathsf{s}}}$ is, for $m\geq 0$,\footnote{Note that `$m$' is a Sobolev order, whereas `$\mathsf{m}$' is a coefficient in $P$.} equipped with the norm  
\begin{equation}
\lVert u \rVert_{H^{m,\mathsf{s}}_{\mathrm{de,sc}}} = \lVert \rho^{-\mathsf{s}} u \rVert_{L^2} + \lVert \Lambda^{m,\mathsf{s}} u \rVert_{L^2}, 
\label{eq:misc_026}
\end{equation}
or one equivalent to it. Here, $\Lambda^{m,\mathsf{s}}$ is an elliptic (not only in the differential sense, but in all decay senses as well) $\smash{\Psi^{m,\mathsf{s}}_{\mathrm{de,sc}}}$-operator. If $\smash{\Lambda^{m,\mathsf{s}}}$ is injective, then the first term on the right-hand side of \cref{eq:misc_026} can be removed, in which case the norm works for $m<0$. 


For the statement of our main theorem, it suffices to consider only $m\in \bbN$. When $m\in \bbN$, \cref{eq:Sobolev_def} may be re-expressed in terms of the $C^\infty(\bbO)$-algebra $\operatorname{Diff}_{\mathrm{de,sc}}$ of de,sc- differential operators:
\begin{equation}
H_{\mathrm{de,sc}}^{m,\mathsf{s}} = \{u\in \calS'(\bbR^{1,d}) : Lu\in L^2(\bbR^{1,d})\text{ for all }L\in \operatorname{Diff}^{m,\mathsf{s}}_{\mathrm{de,sc}}  \}.
\end{equation}
Concretely, $\operatorname{Diff}_{\mathrm{de,sc}}^{m,\mathsf{0}}$ is the $C^\infty(\bbO)$-subalgebra of $\operatorname{Diff}^m(\bbR^{1,d})$ generated by the set $\calV_{\mathrm{de,sc}}$ of \emph{de,sc-vector fields} (see below), and 
\begin{align}
\begin{split} 
\operatorname{Diff}_{\mathrm{de,sc}}^{m,\mathsf{s}} &= \rho^{-\mathsf{s}} \operatorname{Diff}_{\mathrm{de,sc}}^{m,\mathsf{0}}.
\end{split} 
\end{align}
Then, an equivalent norm to the one above is
\begin{equation} \label{eq:def-desc-norm-integer-m}
	\lVert u \rVert_{H^{m,\mathsf{s}}_{\mathrm{de,sc}}} = \lVert \rho^{-\mathsf{s}} u \rVert_{L^2} + 
	\sum_{\Lambda} \lVert \rho^{-\mathsf{s}} \Lambda u \rVert_{L^2},
\end{equation}
where $\Lambda$ runs over all $m$-fold products of generators of $\calV_{\mathrm{de,sc}}$.

The set $\calV_{\mathrm{de,sc}}$  is the $C^\infty(\bbO)$-submodule of $\calV(\bbR^{1,d})$ generated by vector fields of the form $\chi_1 \partial_t,\chi_1 \partial_{x_j}$ for $\chi_1\in C^\infty(\bbM)$ vanishing near null infinity and some other vector fields near null infinity. In order to specify these, let
\begin{itemize}
	\item $\mathrm{nf}$ stand for any one of $\mathrm{nf}\in \{\mathrm{nPf},\mathrm{nFf}\}$, this standing for ``null face,'' 
	\item and $\mathrm{Of}\in \{\mathrm{Pf},\mathrm{Sf},\mathrm{Ff}\}$ stand for one of the other faces (the `O' standing for ``other''). 
\end{itemize}
Then, the remaining vector fields which generate $\calV_{\mathrm{de,sc}}$ are those of the form 
\begin{equation} 
\chi_2 \rho_{\mathrm{nf}}^2 \rho_{\mathrm{Of}} \partial_{\rho_{\mathrm{nf}}} , \chi_2 \rho_{\mathrm{nf}}\rho_{\mathrm{Of}}^2 \partial_{\rho_{\mathrm{Of}}}, \chi_2 \rho_{\mathrm{nf}}^2\rho_{\mathrm{Of}} \partial_{\theta_j}
\end{equation} 
for $j\in \{1,\dots,d-1\}$, where 
\begin{itemize}
	\item $\chi_2 \in C^\infty(\bbO)$ is supported near a single corner $\mathrm{nf}\cap\mathrm{Of}$ of $\bbO$,
	\item we are defining the partial derivatives using a coordinate system $\rho_{\mathrm{nf}},\rho_{\mathrm{Of}},\theta_1,\dots,\theta_{d-1}$ where $\theta_1,\dots,\theta_{d-1}$ denotes a local coordinate chart on $\bbS^{d-1}_{\bfx/r}$.
\end{itemize}
(That is, a vector field $V\in \calV(\bbR^{1,d})$ lies in $\calV_{\mathrm{de,sc}}$ if and only if it is a scattering vector field at $\mathrm{Pf},\mathrm{Sf},\mathrm{Ff}$ and a double edge (de) vector field at $\mathrm{nPf},\mathrm{nFf}$, with the de- and sc- structures being compatible at the corners in the natural sense.) 
An equivalent global definition is 
\begin{multline}
	\calV_{\mathrm{de,sc}} = \operatorname{span}_{C^\infty(\bbO)}\{ \rho_{\mathrm{nPf}}\rho_{\mathrm{nFf}} \partial_t,\rho_{\mathrm{nPf}}\rho_{\mathrm{nFf}} \partial_{x_j},  \rho_{\mathrm{nPf}}^{-1}\rho_{\mathrm{nFf}}^{-1} \chi_3 (\partial_{|t|} + \partial_r), \chi_3 r^{-1}\partial_{\theta_k} \\ :j\leq d\text{ and } k\leq d-1  \}
\end{multline}
for $\chi_3\in C^\infty(\bbM)$ supported away from $\operatorname{cl}_\bbM\{tr=0\}$ (i.e.\ the union of the initial hypersurface, the $t$-axis, and their ``boundary points'' in $\bbM$)\footnote{That way, $\chi_3(\partial_{|t|}+\partial_r)$, $\chi_3 r^{-1} \partial_{\theta_k}$, defined using the coordinates $t,r,\theta_1,\dots,\theta_{d-1}$, are smooth vector fields.} and identically equal to $1$ near null infinity.
In \cite{sussman2023massive}, the pseudodifferential calculus $\Psi_{\mathrm{de,sc}}$ is constructed by ``quantizing'' the Lie algebra $\calV_{\mathrm{de,sc}}$ in local coordinates. Near points in the interiors of $\mathrm{nPf},\mathrm{nFf}$, this gives the double edge calculus. Near points in the interiors of $\mathrm{Pf},\mathrm{Sf},\mathrm{Ff}$, this gives the sc-calculus.

The set 
\begin{equation}
	S \operatorname{Diff}_{\mathrm{de,sc}}^{m,\mathsf{s}}=\operatorname{span}_{S^{\mathsf{0}}(\bbO)}\operatorname{Diff}_{\mathrm{de,sc}}^{m,\mathsf{s}}
\end{equation}
is the $S^{\mathsf{0}}(\bbO)$-module generated by the de,sc- operators of the corresponding order. In other words, it contains de,sc-differential operators with coefficients which are merely symbolic on $\bbO$ (with weight $\rho^{-\mathsf{s}}$), rather than smooth (times $\rho^{-\mathsf{s}}$).

Note that $H_{\mathrm{de,sc}}^{0,\mathsf{0}} = L^2(\bbR^{1,d})$, and
\begin{equation} 
\bigcap_{m\in \bbN,\mathsf{s}\in \bbR^5} H_{\mathrm{de,sc}}^{m,\mathsf{s}} = \calS(\bbR^{1,d}), \quad \bigcup_{m\in \bbZ,\mathsf{s}\in \bbR^5} H_{\mathrm{de,sc}}^{m,\mathsf{s}} = \calS'(\bbR^{1,d}).
\end{equation} 
The order $m$ is the amount of differentiability, and the pentuple $\mathsf{s} = (s_{\mathrm{Pf}},s_{\mathrm{nPf}},s_{\mathrm{Sf}},s_{\mathrm{nFf}},s_{\mathrm{Ff}} )$ lists the decay rate, relative to $L^2$, at each of the five boundary hypersurfaces of $\bbO$. The decay rates are listed in the order: past timelike infinity $\mathrm{Pf}$, past null infinity $\mathrm{nPf}$, spacelike infinity $\mathrm{Sf}$, future null infinity $\mathrm{nFf}$, future timelike infinity $\mathrm{Ff}$. 

The de,sc-vector fields yield a vector bundle ${}^{\mathrm{de,sc}}T \bbO \twoheadrightarrow \bbO$ whose smooth sections are precisely the elements of $\calV_{\mathrm{de,sc}}$. Dualizing yields a new cotangent bundle ${}^{\mathrm{de,sc}}T^* \bbO$. 
The (compact) de,sc-phase space in \cref{eq:desc_phase}
is obtained by radially compactifying the fibers of that bundle. The compactification creates a boundary face $\mathrm{df}$ (``fiber infinity''), which is topologically $\bbS^{d}\times \bbO$; this is the de,sc- cosphere bundle.\footnote{The `d' here stands for ``differential.''} Standard microlocal objects such as wavefront sets, elliptic/characteristic sets, essential supports, etc.\ are subsets of the boundary of this phase space, and symbols are functions on the interior conormal to the boundary. For the purposes of defining microlocalizers, it suffices to restrict attention to \emph{classical} de,sc-symbols of order zero, which are just smooth functions on ${}^{\mathrm{de,sc}}\overline{T}^* \bbO$, where it is important to note that smoothness means smoothness all the way up to and including the boundary. (That is, embedding $\bbO$ in a larger $\bbR^{1+d}$ as the surface of revolution of an octagon, a function on $\bbB\times \bbO$ is smooth if and only if it extends to a smooth function on some open neighborhood in $(\bbR^{1+d})^2$.) 

A classical pseudodifferential operator $A \in \Psi_{\mathrm{de,sc}}^{0,\mathsf{0}}$ is elliptic at a point $q\in \partial ({}^{\mathrm{de,sc}}\overline{T}^* \bbO)$ if and only if its principal symbol \begin{equation} 
\sigma_{\mathrm{de,sc}}^{0,\mathsf{0}}(A) \in C^\infty({}^{\mathrm{de,sc}}\overline{T}^* \bbO)  
\end{equation} 
is nonvanishing there. (Strictly speaking, $\sigma_{\mathrm{de,sc}}^{0,\mathsf{0}}(A)$ is an equivalence class of symbols modulo lower order symbols. A usual abuse of notation is to conflate principal symbols with representatives thereof.)

Let us elaborate on the de,sc- notion of wavefront set, associated with the scale of de,sc-Sobolev spaces. For each $u\in \calS'(\bbR^{1,d})$, its de,sc-wavefront set is a subset
\begin{equation} 
\operatorname{WF}_{\mathrm{de,sc}}^{m,\mathsf{s}}(u) \subseteq \partial({}^{\mathrm{de,sc}}\overline{T}^* \bbO).
\end{equation} 
This measures, microlocally, the obstruction to $u$ lying in $H^{m,\mathsf{s}}_{\mathrm{de,sc}}$:
\begin{equation}
q \notin \operatorname{WF}_{\mathrm{de,sc}}^{m,\mathsf{s}}(u) \quad
\text{ iff } \quad \exists A \in \Psi_{\mathrm{de,sc}}^{0,\mathsf{0}} \text{ that is elliptic at } q \text{ and } Au \in H^{m,\mathsf{s}}_{\mathrm{de,sc}}.
\end{equation}
(And it suffices to restrict to classical $A$ in this definition.)
So, 
\begin{equation} 
u\in H^{m,\mathsf{s}}_{\mathrm{de,sc}} \iff \operatorname{WF}_{\mathrm{de,sc}}^{m,\mathsf{s}}(u)=\varnothing.
\end{equation} 
Regarding the interpretation of this notion of wavefront set:
\begin{itemize}
	\item over the interior of $\bbO$, $\operatorname{WF}_{\mathrm{de,sc}}^{m,\mathsf{s}}$ is just the ordinary wavefront set $\operatorname{WF}^m$, so it detects failures of smoothness, along with the directions in which those failures are most severe;
	\item  over  the interiors of $\mathrm{Pf},\mathrm{Sf},\mathrm{Ff}$, $\operatorname{WF}_{\mathrm{de,sc}}^{m,\mathsf{s}}$ is just the sc-wavefront set $\operatorname{WF}_{\mathrm{sc}}^{m,s_{\mathrm{Pf}}}$, $\operatorname{WF}_{\mathrm{sc}}^{m,s_{\mathrm{Sf}}}$, $\operatorname{WF}_{\mathrm{sc}}^{m,s_{\mathrm{Ff}}}$ respectively. Roughly speaking, this means that it detects failures of decay, and moreover which plane waves $e^{i\tau t + i \xi \cdot \bfx}$ are present in the $t\to\infty$ or $r\to\infty$ asymptotics.
	\item 
	Over the interiors of $\mathrm{nPf},\mathrm{nFf}$, $\operatorname{WF}_{\mathrm{de,sc}}^{m,\mathsf{s}}$ is the de-wavefront set $\operatorname{WF}_{\mathrm{de}}^{m,s_{\mathrm{nPf}}}$, $\operatorname{WF}_{\mathrm{de}}^{m,s_{\mathrm{nFf}}}$ respectively. Like the sc-wavefront, this detects failures of decay, but distinguishes different sorts of oscillations. In the present context, de-wavefront set at null infinity distinguishes different oscillations
	\begin{equation}
		e^{\frac{i}{\varrho}(\zeta v+\xi) + i r \eta\cdot  \theta},\quad \zeta,\xi\in \bbR, \quad v=|t|-r, \quad \varrho = \frac{1}{\sqrt{|t|+r}}
	\end{equation}
	present in radiation. Compared to plane waves, these oscillations are slower along the light cone, faster across the light cone, and comparable in the angular direction $\theta$. See \cite[Rem. 1.8]{sussman2023massive}. It is these sorts of asymptotics seen in the Green's functions of the Klein--Gordon equation (see \cite[Example 1.6]{sussman2023massive}).
\end{itemize}

\section{Proof of main theorem}
\label{sec:proof}

\subsection{Characteristic set and Hamilton flow of $P$ and $P+\lambda$}
We briefly recall the relevant phase space structures; for details see \cite[\S4]{sussman2023massive}, some figures in which complement those here.

The operator $P$ is an element of $\Psi_{\desc}^{2,0}(\bbO)$ with principal symbol 
\begin{equation} 
	p(z,\zeta)=g_z(\zeta,\zeta)+\mathsf{m}^2,
\end{equation} 
where $z=(t,\bfx)\in \bbR^{1,d}$ and $\zeta=(\tau,\xi)$ is the spacetime frequency.
The characteristic set $\Sigma$ of $P$ consists of the set of limit points on $\partial ({}^{\mathrm{de,sc}}\overline{T}^* \bbO)$ of the set where $g_z(\zeta,\zeta)=-\mathsf{m}^2$. At fiber infinity ($\mathrm{df}$) the $\mathsf{m}^2$ term is negligible, so $\Sigma\cap \mathrm{df}$ is the union of the ``points at infinity'' of the dual lightcones over every point of $\bbO$. On the other hand, at finite frequencies over the various faces constituting base infinity, the $\mathsf{m}^2$ term is of the same order as $g_z(\zeta,\zeta)$, so $\Sigma\backslash\mathrm{df}$ is the union of the mass shells over every point of $\partial\bbO$. When $\mathsf{m}>0$, the characteristic set has two connected components: $\Sigma=\Sigma^+\sqcup\Sigma^-$, where $\Sigma^{\pm}$ is the component containing the future ($+$) or past ($-$) dual light cone. In the case $\mathsf{m}=0$, the two sheets intersect at the zero section over $\partial\bbO$.

Since $p$ is real-valued, as long as $\lambda$ is non-real, \emph{the principal symbol of $P+\lambda$ never vanishes at finite (de,sc-) frequencies}. This is the source of the key ellipticity on which our proof is based.
But, at $\mathrm{df}$,  $\lambda$ is lower order. Therefore, the characteristic set of $P+\lambda$ is $\Sigma\cap \mathrm{df}$, this set being independent of $\mathsf{m}$. In particular, note that it has two connected components even for $\mathsf{m}=0$.

The operators $P$ and $P+\lambda$ give rise to the same Hamiltonian vector field $H_p$, which is computed at the beginning of \cite[\S4]{sussman2023massive}. The appropriately rescaled vector field $\mathsf{H}_p$ defined in \cite[\S1]{sussman2023massive} (i.e.\ $H_p$ minimally weighted so as to make it tangent to all of the boundaries of the compactified de,sc- phase space) is not only tangent to the phase-space boundary, it is also tangent to $\Sigma$ (since $H_pp=0$). Moreover, it is smooth up to and including the boundary, modulo conormal terms which vanish at all base infinity faces and do not affect the flow structure nearby. Its flow within $\Sigma$ has a number of vanishing (a.k.a.\ ``radial'') sets, which we describe in the $\mathsf{m}>0$ case:
\begin{itemize}
\item $\calR^{\alpha}_{\beta}$ are the sources/sinks of the flow, located over the timelike caps. Importantly, they lie at finite de,sc-frequency. That is, they do not intersect $\mathrm{df}$. It is these radial sets that are associated with the oscillations present in smooth solutions of the Klein--Gordon equation (with nonzero real mass).
\item $\calN^{\alpha}_{\beta}$ are the sources and sinks of the flow within $\Sigma\cap \mathrm{df}$ (i.e.\ ignoring the fiber radial direction and restricting to $\Sigma\cap\mathrm{df}$). This flow, over the interior of $\bbO$, is just the usual null geodesic flow at fiber infinity. So, 
\begin{equation} 
\calN = \bigcup_{\alpha,\beta = +,-} \calN^{\alpha}_{\beta} \subset \mathrm{df}
\label{eq:N}
\end{equation} 
contains the endpoints of compactified null geodesics (lifted to fiber infinity). In the full characteristic set $\Sigma$, the radial set $\calN$ is a saddle point of the flow, but within $\Sigma\cap \mathrm{df}$ it is a source/sink.
\item $\calC^{\alpha}_{\beta}$, $\calK^{\alpha}_{\beta}$, $\calA^{\alpha}_{\beta}$ are additional radial sets over the corners between null infinity and the timelike/spacelike caps. They are all located at fiber infinity and are saddle points for the flow in both $\Sigma,\Sigma\cap\mathrm{df}$. 
\end{itemize}
Here, $\alpha,\beta$ are signs --- there are four copies of each sort of radial set. The sign in the superscript refers to which of $\Sigma^\pm$ it is in (positive energy vs.\ negative energy), and the subscript to whether the radial set is over future or past timelike/null infinity.
We will use $\calR,\calN,\calC,\calK,\calA$, without decorations, to denote the union over all signs, like in \cref{eq:N}.

The structure of the flow, which we partially described in the previous paragraph, is illustrated in \Cref{fig}, \Cref{fig:grand}.

\begin{figure}[t]
		\begin{center}
			\begin{tikzpicture}[scale=.9]
			\coordinate (one) at (-.413\octheight,\octheight) {};
			\coordinate (two) at (-\octheight,.413\octheight) {}; 
			\coordinate (three) at (-\octheight,-.413\octheight) {};
			\coordinate (four) at (-.413\octheight,-\octheight) {};
			\coordinate (five) at (.413\octheight,-\octheight) {};
			\coordinate (six) at (\octheight,-.413\octheight) {};
			\coordinate (seven) at (\octheight,.413\octheight) {}; 
			\coordinate (eight) at (.413\octheight,\octheight) {};
			\coordinate (oner) at (-.413\octheightr,\octheightr) {};
			\coordinate (twor) at (-\octheightr,.413\octheightr) {}; 
			\coordinate (threer) at (-\octheightr,-.413\octheightr) {};
			\coordinate (fourr) at (-.413\octheightr,-\octheightr) {};
			\coordinate (fiver) at (.413\octheightr,-\octheightr) {};
			\coordinate (sixr) at (\octheightr,-.413\octheightr) {};
			\coordinate (sevenr) at (\octheightr,.413\octheightr) {}; 
			\coordinate (eightr) at (.413\octheightr,\octheightr) {};
			\coordinate (oneint) at (-1.25*.413\octheight,1.25\octheight) {};
			\coordinate (eightint) at (.413*1.25\octheight,1.25\octheight) {};
			\coordinate (fourint) at (-.413*1.25\octheight,-1.25\octheight) {};
			\coordinate (fiveint) at (.413*1.25\octheight,-1.25\octheight) {};
			\begin{scope}[scale=1.2] 
			\fill[fill=gray!20] (-.413\octheightr,\octheightr) -- (-\octheightr,.413\octheightr) -- (-\octheightr,-.413\octheightr) -- (-.413\octheightr,-\octheightr) -- (.413\octheightr,-\octheightr) -- (\octheightr,-.413\octheightr) -- (\octheightr,.413\octheightr) -- (.413\octheightr,\octheightr) -- cycle;
			\end{scope} 
			\draw[fill=gray!10] (oner) -- (twor) -- (threer) -- (fourr) -- (fiver) -- (sixr) -- (sevenr) -- (eightr) -- cycle;
			\draw[fill=gray!20] (one) -- (two) -- (three) -- (four) -- (five) -- (six) -- (seven) -- (eight) -- cycle;
			\draw (one) -- (oner);
			\draw (two) -- (twor);
			\draw (three) -- (threer);
			\draw (four) -- (fourr);
			\draw (five) -- (fiver);
			\draw (six) -- (sixr);
			\draw (seven) -- (sevenr);		
			\draw (eight) -- (eightr);
			\begin{scope}[decoration={
				markings,
				mark=at position 0.51 with {\arrow[scale=1.5,>=latex, color=gray]{>}}}]
			\draw[gray, postaction={decorate}] (-.706\octheight,-.706\octheight) -- (.706\octheight,.706\octheight); 
			\draw[gray, postaction={decorate}] plot[smooth,tension=.1] coordinates {
				(-.636\octheight,-.776\octheight) (.125\octheight,-.125\octheight) (.776\octheight,.636\octheight)};
			\draw[gray, postaction={decorate}] plot[smooth,tension=.1] coordinates {
				(-.776\octheight,-.636\octheight) (-.125\octheight,+.125\octheight) (.636\octheight, .776\octheight)};
			\draw[gray, postaction={decorate}]  plot[smooth,tension=.3] coordinates{ 
				(-.566\octheight,-.846\octheight) (-.456\octheight,-.756\octheight) (.275\octheight,-.275\octheight) (.756\octheight, .456\octheight)  (.846\octheight,.566\octheight) };
			\draw[gray, postaction={decorate}]  plot[smooth,tension=.3] coordinates{ 
				(-.846\octheight,-.566\octheight) (-.756\octheight,-.456\octheight) (-.275\octheight,.275\octheight) (.456\octheight,.756\octheight)  (.566\octheight,.846\octheight) };
			\draw[gray, postaction={decorate}]  plot[smooth,tension=.4] coordinates{ 
				(-.526\octheight,-.886\octheight) (-.146\octheight,-.786\octheight) (.45\octheight,-.45\octheight) (.786\octheight,.146\octheight) (.886\octheight,.526\octheight)    }; 
			\draw[gray, postaction={decorate}]  plot[smooth,tension=.4] coordinates{ 
				(-.886\octheight,-.526\octheight) (-.786\octheight,-.146\octheight) (-.45\octheight,.45\octheight) (.146\octheight,.786\octheight) (.526\octheight,.886\octheight)    }; 
			\draw[postaction={decorate}] (three) -- (two);
			\draw[postaction={decorate}] (threer) -- (three); 
			\draw[postaction={decorate}] (seven) -- (sevenr); 
			\draw[postaction={decorate}] (threer) -- (twor);
			\draw[postaction={decorate}] (fourr) -- (threer);
			\draw[postaction={decorate}] (two) -- (one);
			\draw[postaction={decorate}] (fiver) -- (fourr);
			\draw[postaction={decorate}] (six) -- (seven);
			\draw[postaction={decorate}] (sixr) -- (sevenr);
			\draw[postaction={decorate}] (sevenr) -- (eightr);
			\draw[postaction={decorate}] (eightr) -- (oner);
			\draw[gray, postaction={decorate}]  plot[smooth,tension=.5] coordinates{ (fourint) (-.513\octheightr,-.85\octheightr)  (-.75\octheightr,-.443\octheightr) 	(-.886\octheight,-.526\octheight) };
			\draw[gray, postaction={decorate}]  plot[smooth,tension=.5] coordinates{ 	(.886\octheight,.526\octheight) (.75\octheightr,.443\octheightr)  (.513\octheightr,.85\octheightr)  (eightint)};
			\draw[gray, postaction={decorate}]  plot[smooth,tension=.5] coordinates{ (fiveint) (.413\octheightr,-.65\octheightr)  (.75\octheightr,-.443\octheightr) 	(.886\octheightr,-.526\octheightr) };
			\draw[gray, postaction={decorate}]  plot[smooth,tension=.5] coordinates{ 	(-.886\octheightr,.526\octheightr) (-.75\octheightr,.443\octheightr) (-.413\octheightr,.65\octheightr)  (oneint)};
			\draw[gray, postaction={decorate}] (threer) -- (two);
			\draw[gray, postaction={decorate}] (six) -- (sevenr);
			\draw[gray, postaction={decorate}] (-.2\octheight,-1.25\octheight) -- (five);
			\draw[gray, postaction={decorate}] (.2\octheight,-1.25\octheight) -- (fourr);
			\draw[gray, postaction={decorate}] (one) -- (.2\octheight,1.25\octheight);
			\draw[gray, postaction={decorate}] (eightr) -- (-.2\octheight,1.25\octheight);
			\end{scope}
			\begin{scope}[decoration={
				markings,
				mark=at position 0.51 with {\arrow[scale=1.5,>=latex, color=gray]{>}}}] 
			\draw[postaction={decorate}] (one) -- (eight);
			\draw[postaction={decorate}] (two) -- (one); 
			\draw[postaction={decorate}] (twor) -- (two); 
			\draw[postaction={decorate}] (four) -- (five);
			\draw[postaction={decorate}] (six) -- (sixr); 
			\draw[postaction={decorate}] (five) -- (six);
			\end{scope} 
			\begin{scope}[decoration={
				markings,
				mark=at position 0.75 with {\arrow[scale=1.5,>=latex, color=gray]{>}}}]
			\draw[postaction={decorate}] (fourint) -- (four);
			\draw[postaction={decorate}] (fourint) -- (fourr);
			\draw[postaction={decorate}] (fiveint) -- (five);
			\draw[postaction={decorate}] (fiveint) -- (fiver);
			\draw[postaction={decorate}] (one) -- (oneint);
			\draw[postaction={decorate}] (eight) -- (eightint);
			\draw[postaction={decorate}] (oner) -- (oneint);
			\draw[postaction={decorate}] (eightr) -- (eightint);
			\draw[gray, postaction={decorate}]  plot[smooth,tension=.6] coordinates{ (fourint) (-.63\octheight,-1.1\octheight) (-.526\octheight,-.886\octheight) };
			\draw[gray, postaction={decorate}]  plot[smooth,tension=.6] coordinates{(-.616\octheightr,.796\octheightr) (-.73\octheight,1.19\octheight) (oneint) };
			\draw[gray, postaction={decorate}]  plot[smooth,tension=.6] coordinates{ (fiveint) (.73\octheight,-1.19\octheight) (.616\octheightr,-.796\octheightr) };
			\draw[gray, postaction={decorate}]  plot[smooth,tension=.6] coordinates{(.526\octheight,.886\octheight) (.66\octheight,1.0\octheight) (eightint) };
			\end{scope}
			\begin{scope}[decoration={
				markings,
				mark=at position .5 with {\arrow[scale=1.5,>=latex, color=gray]{>};}}]
				\draw[gray, postaction={decorate}]  plot[smooth,tension=.6] coordinates{ (.886\octheightr,-.526\octheightr)(1.1*\octheightr, -.5\octheightr) (1.15*\octheightr,0) (1.05\octheightr,.5\octheightr)   (.5\octheightr,1.05\octheightr)   (0,1.15*\octheightr) 
					(-.5\octheightr,1.1\octheightr)  (-.616\octheightr,.796\octheightr)};
				\draw[gray, postaction={decorate}]  plot[smooth,tension=.6] coordinates{(.616\octheightr,-.796\octheightr)
					(.5\octheightr,-1.1\octheightr)  (0,-1.15*\octheightr) (-.5\octheightr,-1.05\octheightr)  (-1.05\octheightr,-.5\octheightr) (-1.15*\octheightr,0) (-1.1*\octheightr, .5\octheightr) (-.886\octheightr,.526\octheightr) }; 
			\end{scope}
			\draw[thick, darkblue] (oner) -- (twor);
			\filldraw[color=darkblue] (oner) circle (2pt);
			\filldraw[color=darkblue] (twor) circle (2pt);
			\draw[thick, darkblue] (eight) -- (seven);
			\filldraw[color=darkblue] (eight) circle (2pt);
			\filldraw[color=darkblue] (seven) circle (2pt);
			\draw[thick, darkblue] (four) -- (three);
			\filldraw[color=darkblue] (four) circle (2pt);
			\filldraw[color=darkblue] (three) circle (2pt);
			\draw[thick, darkblue] (fiver) -- (sixr);
			\filldraw[color=darkblue] (fiver) circle (2pt);
			\filldraw[color=darkblue] (sixr) circle (2pt);
			\filldraw[color=darkgreen] (one) circle (2pt);
			\filldraw[color=darkgreen] (five) circle (2pt);
			\filldraw[color=mygreen] (two) circle (2pt);
			\filldraw[color=mygreen] (six) circle (2pt);
			\filldraw[color=mygreen] (threer) circle (2pt);
			\filldraw[color=mygreen] (sevenr) circle (2pt);
			\filldraw[color=darkgreen] (fourr) circle (2pt);
			\filldraw[color=darkgreen] (eightr) circle (2pt);
			\filldraw[color=darkred] (oneint) circle (2pt);
			\filldraw[color=darkred] (eightint) circle (2pt);
			\draw[thick, darkred] (oneint) -- (eightint);
			\filldraw[color=darkred] (fourint) circle (2pt);
			\filldraw[color=darkred] (fiveint) circle (2pt);
			\draw[thick, darkred] (fourint) -- (fiveint);
			\node[color=darkblue] (nm) at (.82\octheightr,-.82\octheightr) {$\calN^+_-$};
			\node[color=darkblue] (np) at (-.82\octheightr,.82\octheightr) {$\calN^+_+$};
			\draw[dashed, mygreen] (threer) -- (-1.2\octheightr,-.41\octheightr) node[left] {$\calK^+_-$};
			\draw[dashed, mygreen] (sevenr) -- (1.2\octheightr,.41\octheightr) node[right] {$\calK^+_-$};
			\draw[dashed, darkgreen] (eightr) -- (+.42\octheightr,+1.3\octheightr) node[right] {$\calC^+_+$};
			\draw[dashed, darkgreen] (fourr) -- (-.42\octheightr,-1.3\octheightr) node[left] {$\calC^+_-$};
			\node[color=darkred] (bb) at (.22\octheightr,-.93\octheightr) {$\calR^+_-$};
			\node[color=darkred] (bc) at (-.22\octheightr,.91\octheightr) {$\calR^+_+$};
			\draw[dotted] (3,.1) -- (4,.1) node[right] {$\mathrm{Sf}\cap\{x>0\}$};
			\draw[dotted] (-3,.1) -- (-4,.1) node[left] {$\mathrm{Sf}\cap\{x<0\}$};
			\draw[dotted] (-2.5,-1.5) -- (-3.5,-2.3) node[left] {$\mathrm{nPf}\cap\{x<0\}$};		
			\draw[dotted] (2,-2) -- (3.5,-2) node[right] {$\mathrm{nPf}\cap\{x>0\}$};
			\draw[dotted] (2.5,1.5) -- (3.5,2.25) node[right] {$\mathrm{nFf}\cap\{x>0\}$};		
			\draw[dotted] (-2,2) -- (-3.5,2) node[left] {$\mathrm{nFf}\cap\{x<0\}$};
			\draw[dotted] (0,-3) -- (0,-4.1) node[right] {$\mathrm{Pf}$};
			\draw[dotted] (0,3) -- (0,4.1) node[right] {$\mathrm{Ff}$};
			\node at (.7,-1.6) {$\mathrm{df}_+$};
			\draw[dashed] (.65\octheightr, -1\octheightr) -- (1\octheightr, -1\octheightr) node[right] {$\mathrm{df}_-$};
			\end{tikzpicture}
			\end{center}
	\caption{The flow of $\mathsf{H}_p$ on $\Sigma^+$, for $\mathsf{m}>0$, in the $d=1$ case when $\Sigma^+$ is the boundary of an octagonal prism --- see \cite{sussman2023massive} for a full description. Only part of $\mathrm{df}\cap \Sigma^+$ is shown, namely the right moving part, $\mathrm{df}_+$ (central octagon), and a neighborhood in the left moving part $\mathrm{df}_-$ of the portion over $\partial \bbO$ (outermost region). (Topologically, $\Sigma^+$ should be pictured in the $d=1$ case as a hollow can with polygonal sides, in which the top and bottom of the can are the two connected components of $\mathrm{df}\cap \Sigma^+ \cong \bbS^0\times \bbO$ and the side of the can consists of one sheet of a hyperbola times spacetime infinity, $\cong [-1,+1]\times \partial \bbO$. A can is difficult to draw, so we puncture it in the base, splay out the sides, and then lay it flat. This is how the figure above is arrived at.) The time axis is vertical and the spatial axis $\bbR_x$ is horizontal. The labels $\bullet\mathrm{f}$ specify which boundary hypersurface of the compactified phase space the labeled portion of $\Sigma^+$ is a subset of. For $d\geq 2$, the same picture depicts the flow in the portion of the characteristic set corresponding to zero angular momentum if we ignore the angular degrees of freedom. The radial sets $\calA$ are located at high angular momentum and are therefore not depicted. See \Cref{fig:grand} for a figure with $\calA$. See \cite[\S4]{sussman2023massive} for the computation of the flow. \emph{One key observation on which this paper is based is that the radial sets {\color{darkred}$\calR$} do not intersect $\mathrm{df}$} (the inner and outer regions in the diagram) in the de,sc-phase space. Contrast with the situation in the sc-cotangent bundle \cite[Fig.\ 1]{sussman2023massive}.}
	\label{fig}
\end{figure}

Our sign convention is that, in the $\mathsf{m}>0$ case, the flow is directed from $\calR^+_-$ to $\calR^+_+$ (i.e. past to future) in $\Sigma^+$ but from $\calR^-_+$ to $\calR^-_-$ (future to past) in $\Sigma^-$.  Particles propagate forwards in time and holes propagate backwards.

\subsection{Regularity theory and proof of the main result}

We are interested in the regularity of solutions $u$ to $Pu+\lambda u=f$ given information about the regularity of $f$. At points away from the characteristic set, we can appeal to microlocal elliptic regularity; since the de,sc-characteristic set of $P+\lambda$ is $\Sigma\cap\mathrm{df}$, we can immediately conclude that
\begin{equation}
\label{eq:elliptic-reg}
\WF_{\desc}^{m,\mathsf{s}}(u)\backslash \WF_{\desc}^{m-2,\mathsf{s}}(Pu+\lambda u)\subseteq \Sigma\cap\mathrm{df}.
\end{equation}
In particular, if $Pu+\lambda u\in \calS$, in which case $\WF_{\desc}^{m-2,\mathsf{s}}(Pu+\lambda u)=\varnothing$, we get $\WF_{\desc}^{m,\mathsf{s}}(u)\subseteq \Sigma\cap\mathrm{df}$. 

Within the characteristic set, we need to use propagation estimates. The operator $P+\lambda$ is not of real principal type, but, since the imaginary part of the principal symbol is just the imaginary part $\Im \lambda$ of $\lambda$, hence has a definite sign, we can still propagate \textit{along/against the Hamiltonian flow of $P$ within $\Sigma = \Sigma(\Re \lambda)$}, the characteristic set of $P+\Re \lambda$;\footnote{When we write $\Sigma$ below, we mean the characteristic set of $P+ \Re \lambda$. If $\Re \lambda+\mathsf{m}^2>0$, $\Sigma^\pm$ refers to the components of this disconnected set, as before. If $\Re \lambda+\mathsf{m}^2\leq 0$, then $\Sigma$ has only one component, if $d\geq 2$. However, we can always restrict attention to an arbitrarily small neighborhood of $\mathrm{df}$, where there are two components as before, and these we can refer to as $\Sigma^\pm$. } the imaginary part means simply that we can only propagate regularity in one direction relative to the Hamiltonian flow: forward for $\Im \lambda>0$ and backward for $\Im \lambda<0$, according to our sign conventions. We will say more about this below. 

Since, due to \cref{eq:elliptic-reg}, $u$ automatically has the desired regularity everywhere except at $\mathrm{df}$, we only need to consider the flow in an arbitrarily small neighborhood of $\mathrm{df}$.
The flow enters and exits fiber infinity at $\calA,\calC,\calK,\calN$. Each of these radial sets is associated with an estimate which says when we can propagate control/singularities through --- these are listed in \cite[\S5]{sussman2023massive}, and we will cite them below.

All radial point estimates use a similar method of proof, regardless of whether the radial sets are pure sources/sinks or saddle points. We will recall this method below. However, the conclusion changes dramatically depending on the sort of radial set. For pure sources/sinks, such as the radial set $\calR$ in \cite[\S6]{sussman2023massive} or the Parenti--Shubin radial sets in \cite{vasy2020essential}, the conclusion (assuming the forcing is nice, say Schwartz) is that one can propagate ``below threshold'' -- meaning weak -- control into a radial set and parlay ``above threshold'' control within the radial set to arbitrarily strong control.
In contrast, radial point estimates at saddle points are propagation-like. They say that we can propagate control along the stable manifold through the radial set to get control along the unstable manifold, or vice versa. The hypotheses of the estimate still include an inequality relating the various Sobolev orders appearing in the estimate; it says that given a finite amount of incoming Sobolev regularity, we only get a corresponding amount of outgoing Sobolev regularity. For instance, the estimate at $\calN$ says that for each extra order of incoming smoothness, we get one extra order of outgoing decay at null infinity, or vice versa. The threshold inequalities at the other saddle points $\calA,\calC,\calK$ are more complicated, because they also relate decay at null infinity to decay at spacelike/timelike infinity, but the basic idea is the same. As long as we have enough incoming control, we can conclude a correspondingly large amount of outgoing control.

For the purpose of concatenating propagation statements, we are free to ignore bicharacteristics (integral curves of $\mathsf{H}_p$) which connect radial sets through the fiber interiors, again because we already know \cref{eq:elliptic-reg}. \emph{Within} fiber infinity, the radial sets $\calN$ are global sources and sinks of the flow (as we have already stated; cf.\ \cite{HintzVasyScriEB}), so they become the initial and final points of propagation. Because they are saddle points from the perspective of the full flow, this propagation does \textit{not} require any Sobolev orders to be above or below an absolute threshold. In other words, instead of using an above-threshold (high-regularity) source/sink estimate to start off the propagation sequence, we can directly propagate regularity into $\calN$ \textit{from the fiber interiors}, where we have elliptic regularity (\cref{eq:elliptic-reg}). Similarly, instead of using a below-threshold (low-regularity) source/sink estimate to complete the sequence, we use another saddle point estimate.

For regularity to propagate through \emph{all} of the saddle points in a component of $\Sigma\cap\mathrm{df}$, the Sobolev orders have to satisfy a system of inequalities. 
Specifically:

\begin{proposition}
\label{thm:propagation}
	Let $\pm \Im (\lambda)>0$ and $u\in\mathcal{S}'$. 
	\begin{enumerate}
	\item If $\WF_{\desc}^{m-1,\mathsf{s}+1}(Pu+\lambda u)\cap\Sigma^{\pm}\cap \mathrm{df}=\varnothing$ for a choice of orders satisfying both the inequalities
	\begin{equation} 
	\begin{cases}
	m>\max\{
1+s_{\mathrm{nFf}},\
1-s_{\mathrm{nFf}}+2s_{\mathrm{Ff}},\ 
1/2+s_{\mathrm{nFf}}-s_{\mathrm{Sf}},\ 
1-s_{\mathrm{nPf}}+2s_{\mathrm{Sf}}
\},\\
	m<\,\min\{
1+s_{\mathrm{nPf}},\ 
1-s_{\mathrm{nPf}}+2s_{\mathrm{Pf}},\ 
1/2+s_{\mathrm{nPf}}-s_{\mathrm{Sf}},\ 
1-s_{\mathrm{nFf}}+2s_{\mathrm{Sf}}
\},
	\end{cases}
	\label{eq:ineq}
	\end{equation} 
	then $\WF_{\desc}^{m,\mathsf{s}}(u)\cap\Sigma^{\pm}\cap\mathrm{df}=\varnothing$.
	
	\item If $\WF_{\desc}^{m-1,\mathsf{s}+1}(Pu+\lambda u)\cap\Sigma^{\mp}\cap \mathrm{df}=\varnothing$ for a choice of orders satisfying both the inequalities
	\begin{equation} 
	\begin{cases}
	m<\,\min\{
1+s_{\mathrm{nFf}},\
1-s_{\mathrm{nFf}}+2s_{\mathrm{Ff}},\ 
1/2+s_{\mathrm{nFf}}-s_{\mathrm{Sf}},\ 
1-s_{\mathrm{nPf}}+2s_{\mathrm{Sf}}
\},\\
	m>\max\{
1+s_{\mathrm{nPf}},\ 
1-s_{\mathrm{nPf}}+2s_{\mathrm{Pf}},\ 
1/2+s_{\mathrm{nPf}}-s_{\mathrm{Sf}},\ 
1-s_{\mathrm{nFf}}+2s_{\mathrm{Sf}}
\},
	\end{cases}
	\end{equation} 
	then $\WF_{\desc}^{m,\mathsf{s}}(u)\cap\Sigma^{\mp}\cap\mathrm{df}=\varnothing$.
	\end{enumerate}
\label{prop:main}
\end{proposition}
\begin{remark}
	There should exist exponentially growing solutions $u$ to $Pu+\lambda u = 0$, to which the conclusion of the previous proposition obviously does not apply. The hypothesis that $u$ be tempered is what is excluding such counter-examples to \Cref{prop:main}.  
\end{remark}
\begin{proof}
By \cref{eq:elliptic-reg} (and the fact that wavefronts are closed), there exists a punctured neighborhood of $\Sigma^{\pm}\cap\mathrm{df}$ (in case one) or $\Sigma^{\mp}\cap\mathrm{df}$ (in case two) disjoint from $\WF_{\desc}^{m,\mathsf{s}}(u)$. This only requires that 
\begin{equation} 
	\WF_{\desc}^{m-2,\mathsf{s}}(Pu+\lambda u)\cap\Sigma^{\bullet}\cap \mathrm{df}=\varnothing,
\end{equation} 
whereas our assumption in the proposition statement is stronger by one order.\footnote{$(m-1,\mathsf{s}+1)$ vs.\ $(m-2,\mathsf{s})$. This is the usual ``loss of one order'' when comparing propagation estimates to elliptic estimates.}

Beyond this starting point, the proposition is a concatenation of \cite[Props.\ 5.6, 5.7, 5.10-5.14]{sussman2023massive}, except applied to $P+\lambda$. The propositions were not stated in \cite{sussman2023massive} in this level of generality, but the proofs apply. Indeed, in this former work it was assumed what amounts to $\mathsf{m}^2 + \Re \lambda>0$, but the structure of $\Sigma$ and $\mathsf{H}_p$ is, sufficiently close to fiber infinity, unchanged if we replace $\mathsf{m}^2 + \Re \lambda$ with another real constant.\footnote{In addition, the constant term does not enter into the calculation of any of the thresholds in any of the estimates.} Moreover, our previous work assumed $\Im \lambda=0$. The only modification in the $\Im \lambda\neq 0$ case is the already mentioned fact that we can only propagate estimates/singularities in one direction along the flow. When propagating in that one direction, the contributions to the estimates coming from $\Im \lambda$ have the sign which allows those terms to be thrown out. This is by now a standard phenomenon going under the name of \emph{complex absorption} --- see \cite{VasyGrenoble} for the version of this argument taking place in the Parenti--Shubin calculus.

For the reader not familiar with these sorts of positive commutator arguments, we provide a sketch. 
\begin{itemize}
	\item Let us now recall how propagation estimates work in this context.  Consider the half of the propagation of singularities theorem which propagates regularity \textit{forward} along the rescaled Hamilton flow $\mathsf{H}_p$. This is the statement that if $u$ has no de,sc-wavefront set of order $(m,\mathsf{s})$ at a point in $\Sigma$, then it also has no wavefront set of order $(m,\mathsf{s})$ at any point downstream, as long as the bicharacteristic segment connecting them does not pass through $\WF_{\desc}^{m-1,\mathsf{s}+1}(Pu+\lambda u)$. The proof is a positive-commutator argument based on the construction of a symbol $a$ such that
	\begin{equation}
	H_pa=-\delta \rho^{2m-2,2\mathsf{s}+2}a^2-b^2+e,\quad  a\in S^{2m-1,2\mathsf{s}+1}_{\mathrm{de,sc}}
	\label{eq:PoS-commutant-symbol}
	\end{equation}
	where $\delta>0$ is a constant, $b\in S^{m,\mathsf{s}}_{\mathrm{de,sc}}$ is elliptic at the downstream point, and $e\in S^{2m,2s}_{\mathrm{de,sc}}$ is supported in a neighborhood of the upstream point. One then quantizes these symbols (in a way preserving essential supports) to get pseudodifferential operators $A,B,E$ satisfying
	\begin{equation}
	i[A,P]=-\delta(\Lambda A)^*(\Lambda A)-B^*B+E+F,
	\label{eq:PoS-commutant-op}
	\end{equation}
	where $\Lambda$ is a quantization of $\rho^{m-1,\mathsf{s}+1}$, $F$ is an error term which is an order lower than $E$. Here, one uses that the principal symbol of $i[A,P]$ is $H_pa$.\footnote{The sign in the identity $\sigma_\bullet(i[A,P])=\pm H_pa$ depends on sign conventions for Hamiltonian vector fields and the quantization map. We follow the conventions of \cite{sussman2023massive}, which differ from those in \cite{vasy2020essential}, leading to the direction of propagation (relative to $H_p$) for a given sign of $\Im\lambda$ being flipped compared to \cite{vasy2020essential}.}
	It can be ensured that $A=\check{A}^2$ for some symmetric $\check{A}$.

	For $u$ in a regular enough weighted Sobolev space, \cref{eq:PoS-commutant-op}, combined with the identity
	\begin{equation}
	-2 \Im \langle Pu, A u \rangle = i\langle Pu,Au \rangle - i\langle Au, Pu \rangle  = 
	\langle i[A,P]u,u \rangle 
	\label{eq:commutator-identity}
	\end{equation}
	(assuming that the integration-by-parts here can be justified), 
	yields
	\begin{equation}
	\lVert Bu\rVert^2 = 2\Im \langle Pu,Au\rangle-\delta\lVert\Lambda Au\rVert^2+\langle Eu,u\rangle+\langle Fu,u\rangle.
	\label{eq:PoS-estimate}
	\end{equation}
	Since we want to assume that $(P+\lambda)u$ is under control, not $Pu$, we rewrite this as 
	\begin{equation}
	\lVert Bu\rVert^2=2 \Im \langle (P+\lambda)u,Au\rangle-\delta\lVert \Lambda Au\rVert^2+\langle Eu,u\rangle+\langle Fu,u\rangle-2\Im(\lambda)\lVert \check{A}u\rVert^2.
	\end{equation}
	Since we are interested in bounding $\lVert Bu \rVert$ above, as long as $\Im \lambda>0$ the last term can be ignored, and we have 
	\begin{equation}
	\lVert Bu\rVert^2\leq 2 \Im \langle (P+\lambda)u,Au\rangle-\delta\lVert \Lambda Au\rVert^2+\langle Eu,u\rangle+\langle Fu,u\rangle.
	\label{eq:misc_p00}
	\end{equation}
	All the terms on the right-hand side are under control (where controlling $\langle Fu,u \rangle$ requires assuming \begin{equation} 
		\WF_{\desc}^{m-1/2,\mathsf{s}-1/2}(u)=\varnothing
	\end{equation} 
	in the propagation region, so in reality an inductive argument is required, gaining one half-order at a time), which leads to control of $\lVert Bu \rVert$.\footnote{Actually, Cauchy--Schwarz is used to estimate $\langle (P+\lambda )u, Au \rangle$ in terms of $C(\varepsilon) \lVert \Lambda^{-1} (P+\lambda) u \rVert^2+ \varepsilon \lVert \Lambda A u \rVert^2$. As long as $\varepsilon>0$ is small enough, we can absorb the  $\varepsilon \lVert \Lambda A u \rVert$ term into the $\delta$-term in \cref{eq:misc_p00} and then throw the combination away (as the combination has the ``right'' sign). This is the purpose of the $\delta$-term.} By the choice of $b$, this implies an estimate of the $H^{m,\mathsf{s}}_{\mathrm{de,sc}}$ norm of $u$ microlocalized near the point at which we want to establish control. 
	So, in summary, as long as $\Im \lambda$ has the ``right'' sign, then the same symbol construction used to prove propagation estimates in the $\lambda\in \bbR$ case also works here.
	
	If instead $\Im \lambda<0$, then the discussion above applies, mutatis mutandis, to propagating control in the opposite direction along the flow.

	\item  Now consider the radial point estimates which propagate control into a saddle point of the flow from the stable directions, i.e. forward along the Hamiltonian flow.

	The proofs are based on constructing a symbol $a$ which instead of \cref{eq:PoS-commutant-symbol} satisfies
	\begin{equation}
	H_pa=-\delta\rho^{2m-2,2 \mathsf{s}+2} a^2-b^2+e_{\mathrm{s}}^2-e_{\mathrm{u}}^2+e_{\mathrm{e}},
	\label{eq:RP-commutant-symbol}
	\end{equation}
	where $b$ is elliptic at the radial set, $e_{\mathrm{s}}$ is supported near the stable manifold in $\Sigma$ of the radial set but away from the unstable manifold (including the radial set itself), $e_{\mathrm{u}}$ is supported near the unstable manifold in $\Sigma$ but away from the stable manifold, and $e_{\mathrm{e}}$ is supported away from $\Sigma$. Quantizing both sides of \eqref{eq:RP-commutant-symbol}, applying them to $u$ and pairing with $u$ then yields
	\begin{equation}
	\lVert B u \rVert^2=2\Im\langle Pu,Au\rangle-\delta \lVert\Lambda Au\rVert^2+\lVert E_{\mathrm{s}}u \rVert^2-\|E_{\mathrm{u}}u\|^2+\langle E_{\mathrm{e}}u,u\rangle+\langle Fu,u\rangle.
	\label{eq:RP-estimate}
	\end{equation}
	Since the goal is to estimate $\lVert Bu \rVert$ from above, the $\lVert E_{\mathrm{u}}u \rVert$ term ``has the right sign'' and can be dropped, and every other term on the right-hand side is under control if we assume control of $u$ at the stable manifold. By the same argument as above, the resulting estimate remains valid for $P+\lambda$ when $\Im \lambda>0$. 
	
	To propagate regularity from the unstable manifold, one instead arranges 
	\begin{equation} 
	H_pa=\delta\rho^{2m-2,2\mathsf{s}+2}a^2+b^2+e_{\mathrm{s}}^2-e_{\mathrm{u}}^2+e_{\mathrm{e}};
	\end{equation} 
	the corresponding sign changes mean that the estimate remains valid for $P+\lambda$ when $\Im \lambda<0$.
\end{itemize}
So, in summary, if $\lambda\in \bbC \backslash \bbR$, then the propagation/radial point estimates in \cite[\S5]{sussman2023massive} apply to $P+\lambda$, as long as we are propagating along the flow if $\Im \lambda>0$ and against the flow if $\Im \lambda<0$.

We now return to the manner in which these estimates are concatenated to yield \Cref{prop:main}. For definiteness, we assume $\Im \lambda>0$ and consider propagation in $\Sigma^+$, i.e. case one; thus, we are propagating control along the arrows in \Cref{fig:grand}. The other cases are analogous. Control on $u$ is propagated in the following order:
\begin{itemize}
\item From $(\Sigma^+\cap\mathrm{nPf})\backslash\mathrm{df}$ (in a neighborhood of $\mathrm{df}$) into $\calN^+_-$, as in \cite[Prop.\ 5.10]{sussman2023massive};
\item From $\calN^+_-$ throughout $(\Sigma^+\cap\mathrm{df}\cap\mathrm{Pf})\backslash \calC^+_-$, $\Sigma^+\cap\mathrm{df}\cap\pi^{-1}\bbO^{\circ}$ (i.e.\ fiber infinity over the interior), and throughout $(\Sigma^+\cap\mathrm{df}\cap\mathrm{nPf})\backslash \mathrm{Sf}$ \textit{except} $\calC^+_-$ and the bicharacteristics connecting it to $\calK^+_-$;
\item From $(\Sigma^+\cap \mathrm{df}\cap \mathrm{Pf})\backslash\calC^+_-$ into $\calC^+_-$, using \cite[Prop.\ 5.13]{sussman2023massive}, thereby concluding regularity in all of $\Sigma^+\cap \mathrm{df}\cap \mathrm{Pf}$;
\item From $\calC^+_-$, along the bicharacteristics connecting it to $\calK^+_-$, thereby concluding regularity in all of $(\Sigma^+\cap\mathrm{df}\cap\mathrm{nPf})\backslash\mathrm{Sf}$;
\item From those bicharacteristics into $\calK^+_-$, using \cite[Prop.\ 5.11]{sussman2023massive};
\item From $\calK^+_-$ and $\calN^+_-\cap\mathrm{Sf}$ throughout $(\Sigma^+\cap\mathrm{df}\cap\mathrm{nPf}\cap\mathrm{Sf})\backslash\calA^+_-$;
\item From $(\Sigma^+\cap\mathrm{nPf})\backslash\calA^+_-$ into $\calA^+_-$, using \cite[Prop.\ 5.7]{sussman2023massive}, thereby concluding regularity in all of $\Sigma^+\cap\mathrm{df}\cap\mathrm{nPf}$;
\item From $\Sigma^+\cap\mathrm{df}\cap\mathrm{nPf}$, throughout $(\Sigma^+\cap\mathrm{df}\cap\mathrm{Sf})\backslash\mathrm{nFf}$;
\item From $(\Sigma^+\cap\mathrm{df}\cap\mathrm{Sf})\backslash\mathrm{nFf}$ into $\calA^+_+$, using \cite[Prop.\ 5.6]{sussman2023massive};
\item From $\calA^+_+$, throughout $(\Sigma^+\cap\mathrm{df}\cap\mathrm{nFf}\cap\mathrm{Sf})\backslash(\calK^+_+\cup\calN^+_+)$;
\item From a neighborhood of $(\Sigma^+\cap\mathrm{Sf})\backslash\calK^+_+$ into $\calK^+_+$, using \cite[Prop.\ 5.12]{sussman2023massive}, thereby concluding regularity in all of $(\Sigma^+\cap\mathrm{df}\cap\mathrm{Sf})\backslash\calN^+_+$;
\item From $(\Sigma^+\cap\mathrm{df}\cap\mathrm{Sf}\cap\mathrm{nFf})\backslash\calN^+_+$, throughout $(\Sigma^+\cap\mathrm{df}\cap\mathrm{nFf})\backslash (\mathrm{Ff}\cup\calN^+_+)$;
\item From the bicharacteristics in $\Sigma^+\cap\mathrm{df}\cap\mathrm{nFf}$ connecting $\calK^+_+$ to $\calC^+_+$ into $\calC^+_+$, using \cite[Prop.\ 5.14]{sussman2023massive};
\item From $\calC^+_+$, throughout $(\Sigma^+\cap\mathrm{df}\cap\mathrm{Ff})\backslash\calN^+_+$, thereby concluding regularity in all of $(\Sigma^+\cap\mathrm{df}\cap(\mathrm{nFf}\cup\mathrm{Ff})) \backslash\calN^+_+$;
\item From $(\Sigma^+\cap\mathrm{df})\backslash\calN^+_+$ into $\calN^+_+$, using \cite[Prop.\ 5.10]{sussman2023massive}, finally concluding regularity in all of $\Sigma^+\cap\mathrm{df}$.
\end{itemize}
(This is the same propagation order used in \cite{HintzVasyScriEB}. Beware that this reference uses different notation.)

The inequality \cref{eq:ineq} is just the conjunction of all of the inequalities in the hypotheses of the cited propositions in \cite[\S5]{sussman2023massive}.
\end{proof}

\begin{figure}
	\includegraphics[scale=1]{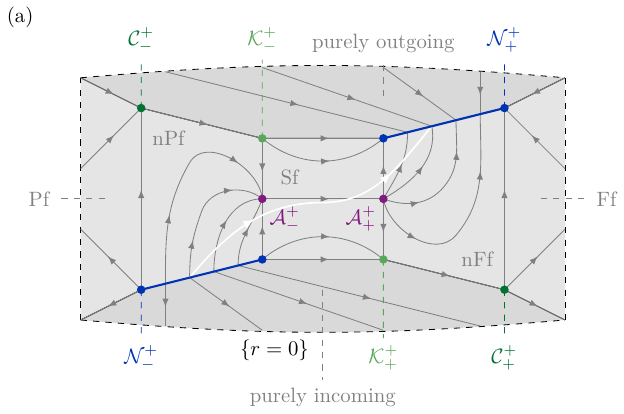}
	\includegraphics{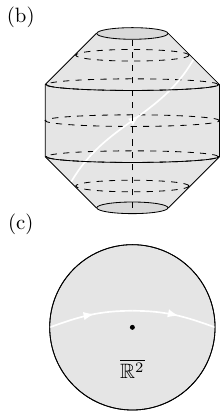}
	\caption[The flow $\mathsf{H}_p$]{(a) The flow, for $g$ the Minkowski metric, in the subset of $\Sigma^+\cap \mathrm{df}$ with nonnegative angular momentum, in the $d=2$ case. \textit{We are not depicting the spatial angular variable $\theta\in \mathbb{S}^1$} and are therefore forgetting the $\partial_\theta$ part of $\mathsf{H}_p$ (which vanishes at null infinity anyways). The top and bottom panels of the figure are the portion of $\Sigma^+\cap \mathrm{df}$ on which the de,sc-angular momentum vanishes. (The de,sc-angular momentum is a rescaling of angular momentum by a function of the spacetime coordinates and is therefore not conserved.) Such momenta are pointing radially inwards (incoming) or radially outwards (outgoing). The horizontal hyperplane in the vertical middle of the figure, in which $\calA$ lies, is the subset of $\Sigma^+\cap \mathrm{df}$ on which angular momentum is maximized. The vertical panels of the figure are over the faces $\mathrm{Pf},\dots,\mathrm{Ff}$ of $\bbO$. The interstitial regions in the figure are the points in  $\Sigma^+\cap \mathrm{df}$ over $\bbO^\circ$ on which angular momentum is nonvanishing. A typical bicharacteristic in the region (white) snakes from $\calN^+_-$ to $\calN^+_+$ without hitting the set $\{r=0\}$ (dashed black). When $g$ is only \emph{asymptotically} flat, the flow over $\bbO^\circ$ is modified, but the flow over $\partial \bbO$ is not. Also depicted are the projection of the white bicharacteristic (b) onto spacetime, i.e.\ ignoring frequency coordinates -- using the same conventions as \Cref{fig:o}(b) -- and (c) onto the radial compactification $\smash{\overline{\mathbb{R}^2}}$ of space, i.e.\ ignoring the time coordinate as well. Since the angular momentum is nonzero, the curve misses the origin $\{r=0\}$ (black); this is also why the curve in (b) does not enter/leave null infinity at $45^\circ$ in the plane of the page. The purely outgoing/incoming bicharacteristics are those that hit the spatial origin.}
	\label{fig:grand}
\end{figure}

\begin{remark}
The fact that we can use elliptic estimates to establish control in the fiber interiors allows us to sidestep a significant technical difficulty in \cite{sussman2023massive} related to the extended radial set $\calN$. 	
In this previous work, there are bicharacteristics in $\Sigma^+$ going from $\calN^+_+$ to $\calK^+_+$ \emph{through the fiber interiors}, then to $\calC^+_+$ and then back to (a different point of) $\calN^+_+$. Consequently, it was required to propagate control through the various radial sets in a different order, using radial point estimates microlocalized to proper subsets of $\calN$.  
\end{remark}

Now we show that Proposition \ref{thm:propagation} implies Theorem \ref{thm:main}. This amounts to showing that the systems of inequalities in that proposition have solutions with all orders arbitrarily high.

\begin{proof}[Proof of Theorem \ref{thm:main}]
By elliptic regularity, $\WF_{\desc}(u)\subset \Sigma\cap\mathrm{df}$.

For any $N>1$, let us take 
\begin{equation} 
s_{\mathrm{Ff}}=s_{\mathrm{nFf}}=N,
\hspace{30pt}
s_{\mathrm{Sf}}=2N,
\hspace{30pt}
s_{\mathrm{Pf}}=s_{\mathrm{nPf}}=4N,
\hspace{30pt}
m=2N.
\end{equation}
Then the inequalities in the first part of Proposition \ref{thm:propagation} are satisfied, so from $f\in\mathcal{S}$ we conclude that $\WF_{\desc}^{m,\mathsf{s}}(u)\cap\Sigma^{\pm}\cap\mathrm{df}=\varnothing$. Taking $N$ arbitrarily large, we conclude that $\WF_{\desc}(u)\cap\Sigma^{\pm}\cap\mathrm{df}=\varnothing$.

Taking instead
\begin{equation} 
s_{\mathrm{Ff}}=s_{\mathrm{nFf}}=4N,
\hspace{30pt}
s_{\mathrm{Sf}}=2N,
\hspace{30pt}
s_{\mathrm{Pf}}=s_{\mathrm{nPf}}=N,
\hspace{30pt}
m=2N,
\end{equation}
by a completely analogous argument we find $\WF_{\desc}\cap\Sigma^{\mp}\cap\mathrm{df}=\varnothing$. Therefore, $u$ has no de,sc-wavefront set anywhere, hence is Schwartz.
\end{proof}

Proposition \ref{thm:propagation} allows us to make a sharper statement about global Sobolev regularity of solutions to $Pu+\lambda u=f$ assuming that $f$ is in a Sobolev space within the range allowed by the inequalities. However, because we are forced to propagate regularity in the same direction with respect to the Hamiltonian flow in both sheets of the characteristic set, hence in \textit{opposite} directions with respect to time, the inequalities required for propagation in the two sheets are incompatible. Therefore, to deduce from \Cref{thm:propagation} a global statement, one needs to use \textit{variable-order} Sobolev spaces, whose orders $m,\mathsf{s}$ are allowed to vary smoothly on ${}^{\desc}\overline{T}^* \bbO$. 
See \cite{VasyGrenoble} for precise definitions in the sc- setting, which directly carry over to the de,sc-setting. 
Here, it suffices to restrict attention to $m,s_\bullet$ constant in a neighborhood of each component of the characteristic set.\footnote{This is different from how variable orders are usually used in microlocal analysis, where the orders vary \emph{within} each component of the characteristic set, so as to be below threshold on one radial set and above threshold on the other radial set in the same component. The reason this is not necessary here (nor in \cite{sussman2023massive}) is that a pentuple $\mathsf{s}\in \bbR^5$ consisting of constant $s_\bullet$ already functions like a variable decay order which jumps discontinuously  when moving from one boundary hypersurface to another.} 
The result is the following:
\begin{theorem}
\label{thm:variable-order}
	Fix \emph{nonreal} $\lambda\in \bbC$, $m\in C^\infty({}^{\desc}\overline{T}^* \bbO ;\bbR)$, and $\mathsf{s}\in C^\infty({}^{\desc}\overline{T}^* \bbO ;\bbR^5)$. Assume that $\Im(\lambda)>0$ and that $m$ as well as each $s_\bullet$ are constant in a neighborhood of each connected component of $\Sigma\cap\mathrm{df}$. Suppose that the following inequalities are satisfied:
	\begin{itemize}
	\item $m<1+s_{\mathrm{nPf}}$ near $\calN^{\pm}_-$ and $m>1+s_{\mathrm{nPf}}$ near $\calN^{\mp}_-$;
	\item $m<1+2s_{\mathrm{Pf}}-s_{\mathrm{nPf}}$ near $\calC^{\pm}_-$ and $m>1+2s_{\mathrm{Pf}}-s_{\mathrm{nPf}}$ near $\calC^{\mp}_-$;
	\item $m>1+2s_{\mathrm{Sf}}-s_{\mathrm{nPf}}$ near $\calK^{\pm}_-$ and $m<1+2s_{\mathrm{Sf}}-s_{\mathrm{nPf}}$ near $\calK^{\mp}_-$;
	\item $m<1/2+s_{\mathrm{nPf}}-s_{\mathrm{Sf}}$ near $\calA^{\pm}_-$ and $m>1/2+s_{\mathrm{nPf}}-s_{\mathrm{Sf}}$ near $\calA^{\mp}_-$;
	\item $m>1/2+s_{\mathrm{nFf}}-s_{\mathrm{Sf}}$ near $\calA^{\pm}_+$ and $m<1/2+s_{\mathrm{nFf}}-s_{\mathrm{Sf}}$ near $\calA^{\mp}_+$;
	\item $m<1+2s_{\mathrm{Sf}}-s_{\mathrm{nFf}}$ near $\calK^{\pm}_+$ and $m>1+2s_{\mathrm{Sf}}-s_{\mathrm{nFf}}$ near $\calK^{\mp}_+$;
	\item $m>1+2s_{\mathrm{Ff}}-s_{\mathrm{nFf}}$ near $\calC^{\pm}_+$ and $m<1+2s_{\mathrm{Ff}}-s_{\mathrm{nFf}}$ near $\calC^{\mp}_+$;
	\item $m>1+s_{\mathrm{nFf}}$ near $\calN^{\pm}_+$ and $m<1+s_{\mathrm{nFf}}$ near $\calN^{\mp}_+$.
	\end{itemize}
	Then, if $u\in \calS'(\bbR^{1,d})$ solves $Pu +\lambda u = f$ for $f\in H_{\desc}^{m,\mathsf{s}}$, we have $u\in H_{\desc}^{m+1,\mathsf{s}-1}$.
\end{theorem}

\begin{remark}
	The hypothesis that $m,s_\bullet$ be constant near each connected component of the characteristic set can likely be removed, as long as they are monotonic (with the ``right'' sign\footnote{This just means that we can propagate strong regularity/decay to weaker regularity/decay but not vice versa.}) under the Hamiltonian flow. However, the proof would necessarily require some technical modifications, like an adaptation of the sharp G{\aa}rding inequality to the de,sc- setting.
\end{remark}

\begin{remark}
Variable $s_\bullet$ is required to make the sharpest possible statements; however, since the relevant systems of inequalities have solutions with all orders arbitrarily high, there are refinements of \Cref{thm:main} that only assume that $f$ lies in some constant order de,sc-Sobolev space. For example, consider what regularity of $f$ is sufficient to guarantee $u\in H_{\desc}^{m,\mathsf{s}}$. We can find constant orders $(m_1,\mathsf{s}_1)$ and $(m_2,\mathsf{s}_2)$ which satisfy the propagation inequalities of \Cref{thm:propagation}.(1), \Cref{thm:propagation}.(2) respectively and such that $m_1,m_2\geq m$ and $\mathsf{s}_1,\mathsf{s}_2\geq \mathsf{s}$. Then for any constant orders $m_0\geq  \max\{m_1,m_2\}$ and $\mathsf{s}_0\geq\max\{\mathsf{s}_1,\mathsf{s}_2\}$, it follows from \Cref{thm:propagation} that $Pu+\lambda u\in H_{\desc}^{m_0-1,\mathsf{s}_0+1}$ implies $u\in H_{\desc}^{m,\mathsf{s}}$.
\end{remark}

\begin{remark}
One can reformulate \Cref{thm:variable-order} as an invertibility statement for $(P+\lambda): \mathcal{X}^{m,\mathsf{s}}\to\mathcal{Y}^{m,\mathsf{s}}$, where
\begin{equation}
\mathcal{X}^{m,\mathsf{s}}=\{u\in H_{\desc}^{m+1,\mathsf{s}-1} : (P+\lambda)u \in H_{\desc}^{m,\mathsf{s}}\},
\hspace{30pt}
\mathcal{Y}^{m,\mathsf{s}}= H_{\desc}^{m,\mathsf{s}}
\end{equation}
with orders satisfying the inequalities above; cf. \cite[Cor.\ 4]{vasy2020essential}. This provides a precise meaning to $(P\pm i\varepsilon)^{-1}$ independently of the functional calculus for self-adjoint operators. In addition, the signs in $\pm i\varepsilon$ distinguishing the Feynman and anti-Feynman propagators correspond to opposite monotonicity and threshold conditions of $m$, $\mathsf{s}$ in \Cref{thm:variable-order}. 
\end{remark}

\section{Astrophysical examples}
\label{sec:examples}
We first discuss a non-radiative example that could have been discussed in \cite{vasy2020essential} but was not. Recall that the Schwarzschild metric $g_{\mathrm{Schw}}=g_{\mathrm{Schw},\frakm}$ with mass $\frakm>0$ is defined by 
\begin{equation}
	g_{\mathrm{Schw}} =  -\Big(1 - \frac{2\frakm}{r} \Big) \dd t^2 +  \Big(1- \frac{2\frakm}{r} \Big)^{-1} \dd r^2 + r^2(\!\dd \theta^2 + (\sin \theta)^2 \dd \phi^2)  
\end{equation}
on $\bbR_t\times \bbR_{r>2\frakm}\times \bbS^2_{\theta,\phi}\hookrightarrow \bbR^{1,3}_{t,\bfx}$ (the ``Schwarzschild exterior'').
\begin{example}
	\label{ex:astro}
	Consider an astrophysical body -- e.g.\ the Earth, or a star -- not so dense so as to form an event horizon. 
	Suppose that it has a finite lifetime, and that rather than losing mass by the emission of some kind of massless radiation, it disintegrates, expelling mass at subluminal velocities. Likewise, it is to be formed by the accretion of matter moving at subluminal velocities.
	A reasonable model for the gravitational field thus generated is a Lorentzian metric $g$ on $\bbR^{1,3}$ such that $g$ is a stationary asymptotically Schwarzschild metric at large distances, $r\gg 1$ (more precisely, near spacelike infinity and null infinity), and asymptotes to the Minkowski metric at large times, $t\gg 1$. More precisely, let us suppose that there exist some $C>C_0>1$ such that
	\begin{enumerate}[label=(\roman*)]
		\item in $r>\max\{r_0,|t|/C\}$, the metric $g$ has the form 
		\begin{equation}
			g =  -\Big(1 - \frac{2\frakm}{r} \Big) \dd t^2 +  \Big(1- \frac{2\frakm}{r} \Big)^{-1} \dd r^2 + r^2(\!\dd \theta^2 + (\sin \theta)^2 \dd \phi^2) + r^{-2} h
			\label{eq:misc_zzz}
		\end{equation}
		where $\frakm \geq 0$ is the mass of the body, $r_0$ is chosen to be bigger than its radius, $(r,\theta,\phi)$ are spherical coordinates on $\bbR^3$, and $h$ is a symmetric two-tensor on spacetime whose coefficients in the usual Cartesian basis are time-independent symbols of order zero on the spatial slices,
		\item in the region $|t| \geq \max\{1,C_0 r\}$, the metric $g$ differs from the Minkowski metric $g_{\bbM}$ by an $O(t^{-1})$ error:
		\begin{equation}
			g - g_{\bbM} \in |t|^{-1} S^0( \bbM; {}^{\mathrm{sc}} \operatorname{Sym}^2 T^* \bbM ) = |t|^{-1} ( S^0(\bbM) \dd t^2 + \dots )
		\end{equation}
		there. (Note that this is compatible with \cref{eq:misc_zzz} on the overlap of the two regions.)
	\end{enumerate}
	See \Cref{fig:examples}(a). 
	The requirement in (ii) that $g$ be a sc- two-tensor near $\mathrm{cl}_\bbO\{r=0\}$ is a disguised finite lifetime condition, because static perturbations of Minkowski are singular at the north/south poles $\mathrm{cl}_\bbO\{r=0\} \cap \mathrm{Ff}$, $\mathrm{cl}_\bbO\{r=0\} \cap \mathrm{Pf}$ of the spacetime, so not sc- tensors. Roughly, we can think of the astrophysical body as forming via the accretion of matter moving at subluminal velocities and disintegrating via a similar process.

	Unfortunately if $\frakm\neq 0$, then $g$ does \emph{not} satisfy \cref{eq:metric_ass_gen}. The reason is that, despite being suppressed by a power of $r$, the tensor
	$r^{-1} (\!\dd t^2 + \dd r^2)$ has the same order of decay as $\dd t^2 - \dd r^2 \in S^{\mathsf{0}}(\bbO; {}^{\mathrm{de,sc}}\operatorname{Sym}^2 T^* \bbO )$ at null infinity when viewed as a de,sc- two-(co)tensor: 
	\begin{equation}
		r^{-1}  (\!\dd t^2 + \dd r^2) \in \varrho_{\mathrm{Pf}} \varrho_{\mathrm{Sf}} \varrho_{\mathrm{Ff}} S^{\mathsf{0}}(\bbO; {}^{\mathrm{de,sc}}\operatorname{Sym}^2 T^* \bbO ) 
	\end{equation}
	(note the absence of $\varrho_{\mathrm{nPf}},\varrho_{\mathrm{nFf}}$). (Closely related points are discussed in \cite[\S1]{sussman2023massive}.)
	So, despite appearances, the terms in \cref{eq:misc_zzz} involving the mass count as large perturbations at null infinity. They show up in the \emph{principal} symbol of $\square_g$. 
	The moral is that, while arbitrary $O(r^{-2})$ terms in the metric (times $\dd t^2$, $\dd t \dd x_j$, $\dd x_j^2$, $\dd x_j \dd x_k$) are fine, arbitrary $O(r^{-1})$ terms are not.
	
	The problem is that \cref{eq:bbO-def} is the ``wrong'' compactification
	of spacetime --- points in the interior of null infinity are distinguished by different values of $v=|t|-r$. But, the physically/geometrically-relevant notion of null infinity is the one where the points in the interior are the limit points of different null geodesics. On the massive spacetimes above, these are distinguished by different values of $v_*=|t|-r_*$, where $r_*$ is the Eddington--Finkelstein tortoise coordinate, $r_* = r + 2 \frakm \ln (r-2\frakm)$. Null geodesics are approximately contained in level sets of $v_*$. (This would be exact if $h=0$. Even if $h\neq 0$, the approximation gets better as $r\to\infty$.) So, following a null geodesic, 
	\begin{equation} 
	v = v_*+ \frakm \ln (r-2\frakm)\to\infty .
	\end{equation} 
	The null geodesics of $g$ therefore tend to the timelike corners of $\bbO$, as depicted in \Cref{fig:examples}(b). \emph{The points in the interior of $\bbO$'s null infinity are not the limit points of different null geodesics.}
	
	However, this can be remedied by a slight modification of the compactification.
	Fix $\psi,\chi \in C^\infty_{\mathrm{c}}(\bbR)$ such that $\psi(s)=\operatorname{sign}(s)$ identically on a small neighborhood of $\{s=\pm 1\}$, $\psi(s)=0$ identically near $s=0$, and $\chi(s)=1$ identically near $s=0$. Fix $\digamma>0$. 
	Now let
	\begin{equation} 
		\tilde{t} = t - 2 \psi\Big(\frac{t}{r}\Big)\chi\Big(\frac{\digamma}{r^2+t^2} \Big) \frakm \ln (r-2\frakm),
	\end{equation} 
	Thus, $\tilde{t}$ is the temporal tortoise coordinate near null infinity, lagging slightly behind $t$ (or ahead of $t$, if $t<0$), but agrees with $t$ elsewhere.  Note that level sets of $v_*$ are equal to level sets of $|\tilde{t}| - r$ near null infinity.
	
	Let $j:\bbR^{1,d}\to\bbR^{1,d}$ be defined by $j:(t,\bfx)\mapsto (\tilde{t},\bfx)$. This is a diffeomorphism of $\bbR^{1,d}$ if $\digamma$ is sufficiently large.\footnote{An alternative compactification involves keeping the temporal coordinate the same and replacing $\bfx$ with $\tilde{\bfx} = r_*\bfx /r$ near null infinity. This also works as far as the rest of our discussion is concerned, but the computations involve more (sufficiently fast decaying to be unimportant) error terms.}
	Now, instead of the basic octagonal compactification $\iota_0:\smash{\bbR^{1,d}}\hookrightarrow \bbO$ in \cref{eq:bbO-def}, consider the composition 
	\begin{equation}
		\iota: \bbR^{1,d}\overset{j}{\hookrightarrow} \bbR^{1,d} \overset{\iota_0}{\hookrightarrow} \bbO,
		\label{eq:smart_comp}
	\end{equation}
	\emph{This} is the correct compactification. 
	Indeed, the pushforward $j_* g$ of $g$ by $j$ satisfies \cref{eq:metric_ass_gen}. The reason is that
	$g_{\mathrm{Schw}}$ can be written, near null infinity, as 
	\begin{equation}
		g_{\mathrm{Schw}} = - \Big(1 - \frac{2\frakm}{r} \Big) \dd v^2_* - 2  \dd v_*\dd r + r^2 (\!\dd \theta^2 + (\sin \theta)^2 \dd \phi^2).
		\label{eq:Vaidya}
	\end{equation}
	Near null infinity, $j_* g_{\mathrm{Schw}}$ is given by replacing $v_*:\bbR^{1,d}\to \bbR$ by $v$ in \cref{eq:Vaidya}. This is the passage to Eddington--Finkelstein coordinates.
	This means that $j_* g_{\mathrm{Schw}}$ differs from the Minkowski metric by $2\frakm r^{-1} \dd v^2$.
	It turns out that $\dd v$ has one order of decay at null infinity as a de,sc-one form:\footnote{This means that the main term in \cref{eq:Vaidya} involving $\dd v_*$ is $-2 \dd v_* \dd r$, not $\dd v_*^2$.}
	\begin{equation}
		\dd v \in \varrho_{\mathrm{nPf}}\varrho_{\mathrm{nFf}} S^{\mathsf{0}}(\bbO\backslash \!\operatorname{cl}_\bbO\{rt=0\};{}^{\mathrm{de,sc}}T^* \bbO ) .
		\label{eq:misc_8j1}
	\end{equation}
	Consequently, $r^{-1} \dd v^2$ has \emph{four} orders of decay at null infinity, and one order at each of $\mathrm{Pf},\mathrm{Sf},\mathrm{Ff}$. So, not only does $j_* g_{\mathrm{Schw}}$ satisfy \cref{eq:metric_ass_gen}, it does so with room to spare.

	If we also assume that $g$ is non-trapping, then $j_* g$ is as well, and so $j_* g$ satisfies all of our requirements.  One then easily deduces 
	the essential self-adjointness of the Klein--Gordon operator $\square_g + \mathsf{m}^2$ and Taira's mapping property from the corresponding results for $\square_{ j_* g} + \mathsf{m}^2 = j_{*}(\square_g + \mathsf{m}^2 )$. (For Taira's mapping property, this uses that a distribution
	$w\in \calD'(\bbR^{1,d})$ is tempered if and only if $j_* w$ is, and likewise with ``tempered'' replaced by ``Schwartz.'')

	For simplicity, we assumed above that the term $h$ in \cref{eq:misc_zzz} is time-independent. However, the same analysis applies if $h$ is time-dependent, as long as it is well-behaved with respect to the the compactification $\iota$ in \cref{eq:smart_comp}: 
	\begin{equation}
		h \in  j^* S^{\mathsf{0}} (\bbO; {}^{\mathrm{de,sc}} \operatorname{Sym}^2 T^* \bbO ).
	\end{equation}
	In other words, $h$ is a zeroth-order symbol on the octagonal compactification constructed using $(\tilde{t},\bfx)$ in place of $(t,\bfx)$. Then, the essential self-adjointness of the Klein--Gordon operator and Taira's mapping property follow, as above.
\end{example}

\tikzset{snake it/.style={decorate, decoration=snake}}
\begin{figure}[t]
	\begin{center}
		\begin{tikzpicture}
			\draw[dashed] (-2,2) -- (2,2) -- (2,-2) -- (-2,-2) -- cycle;
			\node () at (-1.6,1.6) {(a)};
			\fill[lightgray!30] (0,0) circle (1.7);
			\begin{scope}
				\clip  (0,0) circle (1.7);
				\fill[darkred!40] (1.7,1.9) -- (1,1.9) -- (.3,.5) -- (.3,-.5) -- (1,-1.9) -- (1.7,-1.9) -- cycle;
				\fill[darkred!40] (-1.7,1.9) -- (-1,1.9) -- (-.3,.5) -- (-.3,-.5) -- (-1,-1.9) -- (-1.7,-1.9) -- cycle;
				\fill[darkblue, opacity=.3] (-1.6,1.9) -- (-.1,.2) -- (.1,.2) -- (1.6,1.9) -- cycle;
				\fill[darkblue, opacity=.3] (-1.6,-1.9) -- (-.1,-.2) -- (.1,-.2) -- (1.6,-1.9) -- cycle;
				\node[darkred] () at (1.3,0) {(i)};
				\node[darkblue] () at (0,1.3) {(ii)};
				\draw[dashed] (-2,2) -- (2,-2);
				\draw[dashed] (2,2) -- (-2,-2);
			\end{scope}
			\draw (0,0) circle (1.7);
		\end{tikzpicture}
		\begin{tikzpicture}
			\draw[dashed] (-2,2) -- (2,2) -- (2,-2) -- (-2,-2) -- cycle;
			\node () at (-1.6,1.6) {(b)};
			\begin{scope}[scale=.65]
			\coordinate (one) at (.413\octheight,\octheight) {};
			\coordinate (two) at (\octheight,.413\octheight) {}; 
			\coordinate (three) at (\octheight,-.413\octheight) {};
			\coordinate (four) at (.413\octheight,-\octheight) {};
			\coordinate (five) at (-.413\octheight,-\octheight) {};
			\coordinate (six) at (-\octheight,-.413\octheight) {};
			\coordinate (seven) at (-\octheight,.413\octheight) {}; 
			\coordinate (eight) at (-.413\octheight,\octheight) {};
			\draw[fill=gray!10] (one) -- (two) -- (three) -- (four) -- (five) -- (six) -- (seven) -- (eight) -- cycle;
			\node (Ff) at (.1,1.2\octheight) {$\mathrm{Ff}$};
			\node (nFf) at (1.0\octheight,.81\octheight) {$\mathrm{nFf}$};
			\node (Sf) at (1.2\octheight,0) {$\mathrm{Sf}$};
			\node (Pf) at (.1,-1.2\octheight) {$\mathrm{Pf}$};
			\node (nPf) at (1.0\octheight,-.81\octheight) {$\mathrm{nPf}$};
			\draw[orange] plot [smooth] coordinates { (four) (.6\octheight,-.5\octheight) (-.5\octheight,.5\octheight) (eight)};
			\begin{scope}
				\clip (one) -- (two) -- (three) -- (four) -- (five) -- (six) -- (seven) -- (eight) -- cycle;
				\draw[gray, dashed] (-2.1,1.9) -- (2.1,-1.7);
			\end{scope}
			\draw (one) -- (two) -- (three) -- (four) -- (five) -- (six) -- (seven) -- (eight) -- cycle;
			\end{scope} 
		\end{tikzpicture}
		\begin{tikzpicture}
			\draw[dashed] (-2,2) -- (4,2) -- (4,-2) -- (-2,-2) -- cycle;
			\node () at (-1.6,1.6) {(c)};
			\begin{scope}[scale=.7]
				\coordinate (one) at (.413\octheight,\octheight) {};
				\coordinate (two) at (\octheight,.413\octheight) {}; 
				\coordinate (three) at (\octheight,-.413\octheight) {};
				\coordinate (four) at (.413\octheight,-\octheight) {};
				\coordinate (five) at (-.413\octheight,-\octheight) {};
				\coordinate (six) at (-\octheight,-.413\octheight) {};
				\coordinate (seven) at (-\octheight,.413\octheight) {}; 
				\coordinate (eight) at (-.413\octheight,\octheight) {};
				\draw[fill=gray!10] (one) -- (two) -- (three) -- (four) -- (five) -- (six) -- (seven) -- (eight) -- cycle;
				\begin{scope}
					\clip (one) -- (two) -- (three) -- (four) -- (five) -- (six) -- (seven) -- (eight) -- cycle;
					\fill[darkblue, opacity=.15] (-2.2,1.3) to[out=-30,in=210] (2.2,1.3) -- (2.2,-1.3) to[out=150,in=30] (-2.2,-1.3) -- cycle ;
					\fill[darkred, opacity=.3] (-2.2,2.3) to[out=-40,in=220] (2.2,2.3) -- cycle;
					\fill[darkred, opacity=.3] (-2.2,-2.3) to[out=40,in=-220] (2.2,-2.3) -- cycle;
					\fill[gray!10] (.1,-1.2) -- (.1,1) -- (-.1,1) -- (-.1,-1.2) -- cycle;
					\fill[gray!10] (-.3,2.5) -- (.3,2.5) -- (.1,1.2) -- (-.1,1.2) -- cycle;
					\fill[gray!10] (-.3,-2.5) -- (.3,-2.5) -- (.1,-1.2) -- (-.1,-1.2) -- cycle;
					\draw[orange, snake it] (-.1,1) to[out=170, in=-55] (-2,1.8);
					\draw[orange, snake it] (-.1,1.1) to[out=170, in=-55] (-2,1.9);
					\draw[orange, snake it] (.1,1) to[out=10, in=235] (2,2);
					\draw[orange, snake it] (.1,1.1) to[out=10, in=235] (2,2.1);
					\draw[orange, snake it] (-.1,-1) to[out=190, in=55] (-2,-1.8);
					\draw[orange, snake it] (-.1,-1.1) to[out=190, in=55] (-2,-1.9);
					\draw[orange, snake it] (.1,-1) to[out=-10, in=115] (2,-2);
					\draw[orange, snake it] (.1,-1.1) to[out=-10, in=115] (2,-2.1);
				\end{scope}
				\draw (one) -- (two) -- (three) -- (four) -- (five) -- (six) -- (seven) -- (eight) -- cycle;
				\draw[darkred, dashed] (.8,1.9) -- (1.5,1.9) node[right] {$g\cong g_{\mathrm{Schw},M_{\mathrm{F}}}$};
				\draw[darkred, dashed] (.8,-1.8) -- (1.5,-1.8) node[right] {$g=g_{\mathrm{Schw},M_{\mathrm{I}}}$};
				\draw[gray, dashed] (0,-1.8) -- (0,-2.5) node[right] {accretion};
				\draw[gray, dashed] (0,1.8) -- (0,2.5) node[right] {disintegration};
				\node[darkblue!70] () at (3.7,0) {$g\cong g_{\mathrm{Schw},M}$};
			\end{scope} 
		\end{tikzpicture}
	\end{center}
	\caption{(a) The regions (i), (ii) discussed in \Cref{ex:astro}, as seen in $\bbM=\overline{\bbR^{1,3}}$, with the light cone drawn as a dashed ``X''; (b) a null geodesic of Schwarzschild $g_{\mathrm{Schw}}$, in orange, tending to the timelike corners of the compactification \cref{eq:bbO-def}. This means that the notion of null infinity supplied by this compactification is not the physically/geometrically relevant notion of null infinity. We would instead like the geodesic to look like the gray dashed line, as the null geodesics of Minkowski do. This is what the alternative compactification in \cref{eq:smart_comp} accomplishes. Finally, (c) shows the structure of the Vaidya-like metric $g$ discussed in \Cref{ex:vaidya}. The gray region is where the metric is isometric to something other than a Schwarzschild metric. The radiation being emitted is shown. The gray area around $\{r=0\}$ is where the astrophysical body is.
	}
	\label{fig:examples}
\end{figure}

In the previous example, the spacetime had constant mass near null infinity. The mass cannot be converted to energy and lost, except by subluminal effects (e.g. disintegration). In contrast:

\begin{example}[Vaidya-like metrics]
	\label{ex:vaidya}
	We now discuss spacetimes that model the gravitational fields around spherically symmetric astrophysical bodies (with finite lifetimes, and without event horizons) that emit and absorb \emph{null dust}, a certain kind of massless radiation. 
	We begin with the Vaidya metric \cite{Vaidya}. This is the metric on 
	\begin{equation} 
		\bbR_{v}\times (0,+\infty)_r\times\bbS^2_{\theta,\phi} = \bbR_v\times (\bbR^3\backslash \{0\})_\bfx
	\end{equation} 
	given by taking the Schwarzschild metric in Eddington--Finkelstein coordinates (i.e.\ the right-hand side of \cref{eq:Vaidya}, with $v$ in place of $v_*$) and promoting $\frakm$ to a function of $v$: 
	\begin{equation}
		g_{\mathrm{Vaidya}} = - \Big(1 - \frac{2\frakm(v)}{r} \Big) \dd v^2 - 2  \dd v \dd r + r^2 (\!\dd \theta^2 + (\sin \theta)^2 \dd \phi^2).
		\label{eq:real_Vaidya}
	\end{equation}
	The Einstein tensor of $g_{\mathrm{Vaidya}}$ is null, which is why the Vaidya metric is used to describe the gravitational field generated by a spherically symmetric object which emits null dust  \cite[\S9.5]{griffiths2009exact}.
	While \cref{eq:real_Vaidya} guarantees good behavior of outgoing null geodesics, it is not obvious whether the metric is well-behaved in the past. (See \cite{C-N-Vaidya} for an analysis of incoming radial null geodesics under particular assumptions.) Moreover, something about the $v\to\infty$ behavior of $\frakm(v)$ must be assumed if the metric is to be well-behaved at timelike infinity. We assume that, outside of some compact subset of $\bbR_v$, the function $\frakm(v)$ is constant. So, there exist $M\geq M_{\mathrm{F}} >0$ and real numbers $v_1>v_0$ such that  $\frakm(v)=M$ if $v<v_0$ and $\frakm(v) =M_{\mathrm{F}}$ if $v>v_1$.
	So, the object generating the gravitational field is only emitting null dust for a finite amount of time.
	From the formula \cref{eq:Vaidya} for the Schwarzschild metric in Eddington--Finkelstein coordinates, we see that $g_{\mathrm{Vaidya}}$ is isometric to mass-$M$ Schwarzschild in $v<v_0$ and isometric to mass-$M_{\mathrm{F}}$ Schwarzschild in $v>v_1$.  
	
	We can interpret $g_{\mathrm{Vaidya}}$ as a metric on $\bbR^{1,3}_{t,\bfx}\backslash \{r=0\}$ by introducing $t(v,r)=v+r$. In terms of the coordinate system $(t,r,\theta,\phi)$, 
	\begin{equation}
		g_{\mathrm{Vaidya}}  = g_\bbM+ \frac{2\frakm(v)}{r} (\!\dd t - \dd r)^2 = g_\bbM+ \frac{2\frakm(t-r)}{r} (\!\dd t - \dd r)^2,
	\end{equation}
	where $g_\bbM$ is the Minkowski metric. Since $\dd t-\dd r$ is $O(\varrho_{\mathrm{nFf}} )$ as a de,sc- one-form (\cref{eq:misc_8j1}), we conclude that, for any $c\in (0,1)$ and $r_0>0$, 
	\begin{equation}
		g_{\mathrm{Vaidya}}-g_\bbM \in \varrho_{\mathrm{Sf}}\varrho_{\mathrm{nFf}}^4\varrho_{\mathrm{Ff}} C^\infty(\bbO \cap \{t>-cr-1, r>r_0\} ; {}^{\mathrm{de,sc}}\operatorname{Sym}^2 T^* \bbO ).
		\label{eq:misc_053}
	\end{equation}
	Importantly, 
	\begin{equation} 
			g_{\mathrm{Vaidya}} = g_\bbM + 2 M r^{-1} (\!\dd t-\dd r)^2
	\end{equation} 
	in some neighborhood in $\bbO$ of $\mathrm{cl}_\bbO \{t=0\} \backslash \{r<r_0\}$, $r_0\gg 1$. So, $g_{\mathrm{Vaidya}}$ is isometric to mass-$M$ Schwarzschild there (via Eddington--Finkelstein coordinates). We can therefore glue together one copy of the Vaidya spacetime and one time-reversed copy (with some mass $M_{\mathrm{I}}\leq M$ in place of $M_{\mathrm{F}}$). The resulting spacetime $(\bbR^{1,d},g_0)$ is (outside of some big ball $\{r<r_0\}$) isometric to mass-$M_{\mathrm{I}}$ Schwarzschild near $\mathrm{Pf}$, mass-$M$ Schwarzschild near $\mathrm{Sf}$, and mass-$M_{\mathrm{F}}$ Schwarzschild near $\mathrm{Ff}$. It follows from \cref{eq:misc_053} that 
	\begin{equation}
		g_{0}-g_\bbM \in \varrho_{\mathrm{Pf}}\varrho_{\mathrm{nPf}}^4 \varrho_{\mathrm{Sf}}\varrho_{\mathrm{nFf}}^4\varrho_{\mathrm{Ff}} C^\infty(\bbO \cap \{r>r_0\} ; {}^{\mathrm{de,sc}}\operatorname{Sym}^2 T^* \bbO ).\label{eq:misc_054}
	\end{equation}
	
	Now, $g_0$ is singular at $r=0$, but this should not bother us, since the Vaidya metric is only to be taken seriously as a physical model for $r\gg 1$, because it does not describe the metric inside of the object that is emitting the null dust. So, let $g$ denote a metric on $\bbR^{1,3}$ that is equal to $g_0$ outside of some neighborhood in $\bbO$ of $\mathrm{cl}_\bbO \{r=0\}$ not intersecting $\mathrm{nPf}\cup\mathrm{Sf}\cup \mathrm{nFf}$. 
	Moreover, suppose that, away from null infinity, $g$ differs from $g_\bbM$ by an $O((r^2+t^2)^{-1/2})$ scattering metric. (That is, its coefficients when written in Cartesian coordinates are elements of $S^{-1}(\bbM)$, smooth at the boundary.\footnote{The smoothness assumption is not necessary here.}) 
	Then, \cref{eq:misc_054} implies 
	\begin{equation}
		g-g_\bbM \in \varrho_{\mathrm{Pf}}\varrho_{\mathrm{nPf}}^4 \varrho_{\mathrm{Sf}}\varrho_{\mathrm{nFf}}^4\varrho_{\mathrm{Ff}} C^\infty(\bbO ; {}^{\mathrm{de,sc}}\operatorname{Sym}^2 T^* \bbO ).\label{eq:misc_055}
	\end{equation}
	So, $g$ is asymptotically Minkowski in the sense that we require.
	That is, \cref{eq:metric_ass_gen} holds. The essential self-adjointness of the Klein--Gordon operator and Taira's mapping property follow, assuming that $g$ satisfies the non-trapping assumption.
	
	The structure of the Vaidya-like metric $g$ is depicted in \Cref{fig:examples}(c). 
\end{example}

Let us close by remarking that we are forced to impose finite lifetime assumptions in \Cref{ex:astro}, \Cref{ex:vaidya} because we cannot, using the microlocal tools discussed in \S\ref{sec:statement}, \S\ref{sec:proof},  handle asymptotically static spacetimes that are not asymptotically Minkowski. This is precisely the sort of assumption that the work of Baskin--Doll--Gell-Redman in \cite{3bodysc,3bodysc-Feynman} should dispose of when combined with our de,sc-analysis here.

\appendix

\section{Deficiency index theory as applied to $\square_g$}
\label{sec:deficiency}
To complete our exposition, we include the proof of the following elementary proposition, which is used to deduce \Cref{thm:main_0} from Taira's mapping property, \Cref{thm:main}:
\begin{proposition}
	Let $g,P$ be as in \S\ref{sec:statement}. Then, $P$ is essentially self-adjoint on $C_{\mathrm{c}}^\infty(\bbR^{1,d})$ with respect to the $L^2(\bbR^{1,d},g)$-inner product if and only if the only $u\in L^2$ satisfying $Pu=\pm iu$ is $u=0$, for both choices of sign. 
	\label{prop:deficiency}
\end{proposition}
\begin{proof}
	By \cite[Chp.\ VIII \S2]{RS1}, essential self-adjointness is equivalent to $\operatorname{Ran}_{C_{\mathrm{c}}^\infty}(P\pm i)$ being dense in $L^2$ for both choices of sign, which is equivalent to the subspace \begin{equation} 
		\operatorname{Ran}_{C_{\mathrm{c}}^\infty}(P\pm i)^\perp \subset L^2
	\end{equation} 
	being $\{0\}$.

	For any $w\in \calD'(\bbR^{1,d})$ and $v\in C_{\mathrm{c}}^\infty(\bbR^{1,d})$, we can interpret $\langle w,v \rangle_{L^2}=\langle v,w\rangle^*_{L^2}\in \bbC$ as a distributional pairing, 
	\begin{equation}
		\langle w,v \rangle_{L^2} = w(v^*).
	\end{equation}
	The map $\langle -,- \rangle : \calD'(\bbR^{1,d})\times C_{\mathrm{c}}^\infty(\bbR^{1,d})\to \bbC$ is 
	jointly continuous.
	This implies (by the symmetry of $P$ on test functions and the density of test functions) that 
	\begin{equation} 
		\langle Pw,v \rangle_{L^2}= \langle w,Pv \rangle_{L^2}
	\end{equation} 
	for all $w\in \calD'(\bbR^{1,d})$ and $v\in C_{\mathrm{c}}^\infty(\bbR^{1,d})$.
	Then, we can compute  
	\begin{align}
		\begin{split} 
			u\in \operatorname{Ran}_{C_{\mathrm{c}}^\infty}(P\pm i)^\perp &\iff \langle u, (P\pm i) v \rangle_{L^2} = 0\text{ for all }v\in C_{\mathrm{c}}^\infty(\bbR^{1,d}) \\ 
			&\iff \langle (P\mp i) u,  v \rangle_{L^2} = 0\text{ for all }v\in C_{\mathrm{c}}^\infty(\bbR^{1,d})  \\
			&\iff (P\mp i)u=0.
		\end{split} 
	\end{align}
	So, $\operatorname{Ran}_{C_{\mathrm{c}}^\infty}(P\pm i)^\perp = \{0\}$ if and only if the only $u\in L^2$ satisfying $Pu=\pm iu$ is $u=0$, for both choices of sign.
\end{proof}

\section{Proof of Taira's mapping property for the exact Minkowski case}
\label{sec:constant_coeff}

For the exact Minkowski D'Alembertian $\square=\partial_t^2 - \smash{\sum_{j=1}^d \partial_{x_j}^2} \in \operatorname{Diff}^2(\bbR^{1,d}_{t,x})$, Taira's mapping property says that, if $\lambda \in \bbC\backslash \bbR$, then, if $f\in \calS(\bbR^{1,d})$,  
\begin{equation}
	u\in \calS'(\bbR^{1,d}), \square u = \lambda u + f \Longrightarrow u\in \calS(\bbR^{1,d}).
	\label{eq:Taira_Mink}
\end{equation}

\begin{proof}[Proof of \cref{eq:Taira_Mink}]	
	Letting $\calF u(\tau,\xi)$ denote the Fourier transform of $u$, where $\tau\in \bbR$ is the coordinate dual to $t$ and $\xi\in \bbR^d$ is dual to $x=(x_1,\dots,x_d)$, $\square u = \lambda u + f$ is equivalent to $-(\tau^2 - \xi^2 + \lambda) \calF u = \calF f(\tau,\xi)$. So, 
	\begin{equation}
		u(t,x) =- \calF^{-1}_{\tau\to t,\xi \to x} \Big( \frac{\calF f(\tau,\xi)}{\tau^2-\xi^2 + \lambda}  \Big) ;
		\label{eq:misc_044}
	\end{equation}
	note that the division by $\tau^2-\xi^2 + \lambda$ is well-defined because $\lambda\notin \bbR$. 
	Because $f\in \calS(\bbR^{1,d})$, we also have $\calF f\in \calS(\bbR^{1,d})$. It follows that $(\tau^2 - \xi^2 + \lambda)^{-1} \calF f$ is Schwartz. Indeed, hitting this with a constant-coefficient differential operator, the result is a linear combination of functions of the form 
	\begin{equation}
		\frac{p(\tau,\xi) \partial_\tau^j \partial_\xi^\alpha  \calF f }{(\tau^2-\xi^2+\lambda)^k } 
	\end{equation} 
	for some $j,k\in \bbN$, multi-index $\alpha$, and polynomial $p$. Since $\lambda\notin \bbR$, we have, for any $N\in \bbR$, 
	\begin{equation}
		\Big\lVert (1+\tau^2+\xi^2)^N \frac{p(\tau,\xi) \partial_\tau^j \partial_\xi^\alpha  \calF f }{(\tau^2-\xi^2+\lambda)^k }  \Big\rVert_{L^\infty} \leq \frac{1}{|\Im \lambda|^k } \lVert (1+\tau^2+\xi^2)^{N_0}  \partial_\tau^j \partial_\xi^\alpha  \calF f  \rVert_{L^\infty} < \infty ,
	\end{equation}
	for some $N_0$. So, $(\tau^2 - \xi^2 + \lambda)^{-1} \calF f$ is indeed Schwartz.

	So, if follows from \cref{eq:misc_044} that $u$ is Schwartz.
\end{proof}



\printbibliography

\end{document}